\documentclass[12pt]{article}
\usepackage{latexsym,amsmath,amsthm,amssymb,amscd,amsfonts}
\usepackage[usenames]{color}

\newtheorem{lemma}{Lemma}[section]
\newtheorem{proposition}[lemma]{Proposition}
\newtheorem{remark}[lemma]{Remark}
\newtheorem{theorem}[lemma]{Theorem}
\newtheorem{definition}[lemma]{Definition}
\newtheorem*{mydef1}{Definition I}
\newtheorem*{mydef2}{Definition II}
\newtheorem{corollary}[lemma]{Corollary}

\newtheorem*{remark*}{Remark}

\newcommand{\Id}{{{\mathchoice {\rm 1\mskip-4mu l} {\rm 1\mskip-4mu l}
      {\rm 1\mskip-4.5mu l} {\rm 1\mskip-5mu l}}}}

\makeatletter \@addtoreset {equation}{section}

\renewcommand\theequation
  {\ifnum \c@subsection>\z@ \arabic{section}.\arabic{subsection}.\arabic{equation}
  \else \arabic{section}.\arabic{equation} \fi}
\makeatother

\newcommand{\LL}{{\mathcal L}}

\newenvironment{packed_enum}{
\begin{enumerate}
  \setlength{\itemsep}{2pt}
  \setlength{\parskip}{5pt}
  \setlength{\parsep}{10pt}
}{\end{enumerate}}

\begin{document}


\title{On the Uniqueness of Hofer's Geometry}

\author{Lev Buhovsky, \ Yaron Ostrover}
\date{\today}
\maketitle
\begin{abstract}
We study the class of norms on the space of smooth functions on a
closed symplectic manifold, which are invariant under the action of
the group of Hamiltonian diffeomorphisms. Our main result shows that
any such norm that is continuous with respect to the
$C^{\infty}$-topology, is dominated from above by the
$L_{\infty}$-norm. As a corollary, we obtain that any bi-invariant
Finsler pseudo-metric on the group of Hamiltonian diffeomorphisms
that is generated by an invariant norm that satisfies the
aforementioned continuity assumption, is either identically zero or
equivalent to Hofer's metric.
 \end{abstract}

\section{Introduction and Main Results}

A remarkable fact, which is among the cornerstones of symplectic
rigidity theory, is that the group of Hamiltonian diffeomorphisms of
a symplectic manifold can be equipped with an intrinsic geometry
given by a bi-invariant Finsler metric known as Hofer's metric. In
contrast with finite-dimensional Lie groups, the existence of such a
metric on an infinite-dimensional group of transformations is highly
unusual due to the lack of compactness. In the past twenty years,
Hofer's metric has been intensively studied with many new
discoveries covering a wide range of aspects in Hamiltonian dynamics
and symplectic geometry.

The purpose of this note is to show that under some mild 
assumption, Hofer's metric is, in a sense, the only bi-invariant
Finsler metric on the group of Hamiltonian diffeomorphisms of closed
symplectic manifolds. In order to state this result precisely we
proceed with some standard definitions and notations, and refer the
reader to the books~\cite{HZ,MS,P1} for symplectic preliminaries,
and further discussions on the group of Hamiltonian diffeomorphisms
and Hofer's geometry.

Let $(M,\omega)$ be a closed $2n$-dimensional symplectic manifold,
and denote by $C^{\infty}_0(M)$ the space of smooth functions that
are zero-mean normalized with respect to the canonical volume form
$\omega^n$. For every smooth time-dependent Hamiltonian function $H
: M \times [0,1] \rightarrow {\mathbb
R}$, 
we associate a  vector field $X_{H_t}$ 
via the equation $i_{X_{H_t}} \omega = - dH_t$, where $ H_t(x) =
H(t,x)$. The flow
of  $X_{H_t}$ 
is denoted by $\phi_H^t$ and is defined for all $t \in [0,1]$. The
main object of this note is the group of Hamiltonian
diffeomorphisms, which consists of all the time-one maps of such
Hamiltonian flows i.e.,
$$ {\rm Ham}(M,\omega) = \{ \phi_H^1 \ | \ {\rm where} \ \phi_H^t \
{\rm is \ a \ Hamiltonian \ flow  } \}$$ When equipped with the
standard $C^{\infty}$-topology, the group ${\rm Ham}(M,\omega)$ 
is an infinite-dimensional Fr\'echet Lie group, whose Lie algebra
${\cal A}$ can be identified with the space $C^{\infty}_0(M)$.
Moreover, the adjoint action of Ham$(M,\omega)$ on ${\cal A}$ is the
standard action of diffeomorphisms on functions i.e., $Ad_\phi f = f
\circ \phi^{-1}$, for every $f \in {\cal A}$ and $\phi \in$
Ham$(M,\omega)$. Next, we define a Finsler (pseudo) distance on
Ham$(M,\omega)$. Given any norm $\| \cdot \|$ on the Lie algebra
${\cal A}$, we define the length of a path $\alpha : [0,1]
\rightarrow {\rm Ham}(M,\omega)$ as
$$ {\rm length}\{ {\alpha}\} = \int_0^1 \| \dot \alpha \| dt =
\int_0^1 \| H_t \| dt ,$$ where $H_t(x)=H(t,x)$ is the unique
normalized Hamiltonian function generating the path $\alpha$. Here
$H$ is said to be normalized if $\int_M H_t \omega^n=0$ for every
$t\in [0,1]$. The distance between two Hamiltonian diffeomorphisms
is given by $$ d(\psi,\varphi) := \inf {\rm length} { \{ \alpha \}
},$$ where the infimum is taken over all Hamiltonian paths $\alpha$
connecting $\psi$ and $\varphi$. It is not hard to check that $d$ is
non-negative, symmetric and satisfies the triangle inequality.
Moreover, a norm on ${\cal A}$ which is invariant under the adjoint
action yields a bi-invariant pseudo-distance function, i.e.
 $d(\psi,\phi) = d(\theta \, \psi,\theta \, \phi) = d(\psi \, \theta ,\phi \, \theta)$
 for every  $\psi,  \phi,  \theta \in {\rm Ham} (M,\omega)$.
From now on we will deal solely with such norms\footnote{We remark that a fruitful study of right-invariant Finsler
metrics on Ham$(M,\omega)$, motivated in part by applications to hydrodynamics, was initiated in a well known paper
by Arnold~\cite{Ar} (see also~\cite{AK},~\cite{KW} and the references within). Moreover, non-Finslerian bi-invariant
metrics on Ham$(M,\omega)$ have been intensively studied in the realm of symplectic geometry, starting with the works of
Viterbo~\cite{Vit}, Schwarz~\cite{Sch}, and Oh~\cite{Oh}, and followed by many others.} and we will refer to $d$ as
the pseudo-distance generated by the norm $\| \cdot \|$.

\begin{remark} \label{Rmk-about-continuity} {\rm
When one studies the geometric properties of the group of Hamiltonian diffeomorphisms, it is convenient to consider smooth paths $ [0,1] \rightarrow {\rm Ham}(M,\omega) $, among which, those that start at the identity correspond to smooth Hamiltonian flows. Moreover, for a given Finsler metric on $ {\rm Ham}(M,\omega) $, a natural assumption from a geometric point of view is that every smooth path $ [0,1] \rightarrow {\rm Ham}(M,\omega) $ is of a finite length. As it turns out, the latter assumption is equivalent to the continuity of the norm on ${\cal A}$ corresponding to the Finsler metric in the $ C^{\infty} $-topology\footnote{We thank A. Katok for his illuminating remark regarding
the naturalness of the assumption that the norm is continuous in the $ C^\infty $-topology.}.
We prove this fact in the Appendix to the paper. Throughout the text we shall consider only such norms.

}
\end{remark}

It is highly non-trivial to check whether a distance function generated by such a norm,
is non-degenerate, that is $d({\mathbb \Id},\phi)
> 0$ for $\phi \neq {\mathbb \Id}$. In fact, for closed
symplectic manifolds, a bi-invariant pseudo-metric $d$ on
Ham$(M,\omega)$ is either a genuine metric or identically zero. This
is an immediate corollary of a well known theorem by
Banyaga~\cite{B}, which states that Ham$(M,\omega)$ is a simple
group, combined with the fact that the null-set $${\rm null}(d) = \{
\phi \in {\rm Ham}(M,\omega) \ | \ d(\Id,\phi) = 0 \}$$ is a normal
subgroup of Ham$(M,\omega)$. A distinguished result by
Hofer~\cite{H} states that the $L_{\infty}$-norm on ${\cal A}$ gives
rise to a genuine distance function on Ham$(M,\omega)$ known as
Hofer's metric. This was discovered and proved by Hofer for the case
of ${\mathbb R}^{2n}$, then generalized by Polterovich~\cite{P}, and
finally proven in full generality by Lalonde and McDuff~\cite{LM}.
In a sharp contrast to the above, Eliashberg and
Polterovich~\cite{EP} showed that for $1 \leq p < \infty$,  the
pseudo-distances on ${\rm Ham}(M,\omega)$ corresponding to the
$L_p$-norms on ${\cal A}$ vanishes identically. A considerable
generalization of the latter result was given by
Ostrover-Wagner~\cite{OW} who proved that for a closed symplectic
manifold:
\begin{theorem}[Ostrover-Wagner~\cite{OW}] \label{OW-theorem}
Let $ \| \cdot \|$ be a Ham$(M,\omega)$-invariant norm on ${\cal A}$
such that $\| \cdot \| \leq C\| \cdot \|_{\infty}$ for some constant
$C$, but the two norms are not equivalent. Then the associated
pseudo-distance $d$ on Ham$(M,\omega)$ vanishes identically.
\end{theorem}

In~\cite{EP}, the authors started a discussion regarding the uniqueness of Hofer's metric (cf.~\cite{Eli},~\cite{P1}). For the case of closed symplectic manifolds, one question they arose is:

\noindent {\bf Question:} Does there exist a Finsler bi-invariant metric on
Ham$(M,\omega)$ which is not equivalent to Hofer's metric.

In this paper we provide an answer to the above question 
under the natural continuity assumption mentioned in Remark~\ref{Rmk-about-continuity}.
More precisely, our main result is:

\begin{theorem}  \label{Main-thm} Let $(M,\omega)$ be a closed symplectic manifold.
Any Ham$(M,\omega)$-invariant pseudo norm $\| \cdot \|$ on
${\mathcal A}$ that is continuous in the $C^{\infty}$-topology, is
dominated from above by the $L_{\infty}$-norm i.e., $\| \cdot \|
\leq C \| \cdot \|_{\infty}$ for some constant $C$.
\end{theorem}

Combining together Theorem~\ref{Main-thm} and
Theorem~\ref{OW-theorem}, we conclude that:

\begin{corollary} For a closed symplectic manifold $(M,\omega)$, any bi-invariant Finsler pseudo metric on Ham$(M,\omega)$,
obtained by a pseudo norm $\| \cdot \|$ on ${\mathcal A}$ that is
continuous in the $C^{\infty}$-topology, is either identically zero
or equivalent\footnote{Here two metrics $d_1,d_2$ are said to be
equivalent if  ${\frac 1 C} \, d_1 \leqslant d_2 \leqslant C d_1$
for some constant $C>0$.} to Hofer's metric. In particular, any non-degenerate bi-invariant Finsler metric on Ham$(M,\omega)$, which is generated by a norm that is continuous in the $C^{\infty}$-topology, gives rise to the same topology on Ham$(M,\omega)$ as the one induced by Hofer's metric.
\end{corollary}

\begin{remark}  {\rm Let us emphasize that any norm $\| \cdot \|$ on ${\mathcal A}$
can be turned into a Ham$(M,\omega)$-invariant pseudo-norm via the invariantization procedure $ \| f \| \mapsto \| f
\|_{inv}$, where:
$$ \| f \|_{inv} =  \inf \Bigl \{  \sum \|\phi_i^*  f_i \|  \ ; \ f = \sum f_i, \ {\rm and \ } \phi_i \in Ham(M,\omega)  \Bigr \} $$
Note that $\| \cdot \|_{inv} \leq \| \cdot \|$.  Thus, if $\| \cdot \|$ is continuous in the $C^{\infty}$-topology,
then so is $\| \cdot \|_{inv}$. Moreover if $\| \cdot \|'$ is a Ham$(M,\omega)$-invariant norm, then:
$$ \| \cdot \|' \leq \| \cdot \| \Longrightarrow \| \cdot \|' \leq \| \cdot \|_{inv}$$
In particular, the above invariantization procedure provides a plethora of Ham$(M,\omega)$-invariant genuine norms
on ${\mathcal A}$, e.g., by taking the homogenization of the $\| \cdot \|_{C^k}$-norms. }
\end{remark}

\noindent{\bf Structure of the paper:} In Section~\ref{section-outline} we sketch an outline of the proof of
Theorem~\ref{Main-thm}. In Section~\ref{section-local-version} we prove a local version of this theorem, which would
serve as the main ingredient in the proof of the general case given in Section~\ref{section-proof-of-thm}.

\noindent {\bf Notations:} Let $x_1,\ldots,x_n$ be the Cartesian
coordinates in ${\mathbb R}^n$. For any multi-index $\alpha =
(\alpha_1,\ldots,\alpha_n)$, set $\partial^{\alpha} =
\partial_1^{\alpha_1} \partial_2^{\alpha_2} \ldots
\partial_n^{\alpha_n},$  where $\partial_i = \partial/\partial x_i$.
For an open set $\Omega \subset {\mathbb R}^n$ 
we denote $C_c(\Omega)$ the space of compactly supported continuous
functions on $\Omega$, and let $\| \cdot \|_{\infty}$ stands for the
$L_{\infty}$-norm. For an integer $k$, define $C^k_c(\Omega)$ the
class of functions $f$ from $C_c(\Omega)$ such that
$\partial^{\alpha}f \in C_c(\Omega)$ for all $|\alpha|\leq k$. The
$C^k$-norm of $u \in C^k_c(\Omega)$ is given by  $$ \| u \|_{C^k} =
\max_{|\alpha|\leq k} \, \sup_{\Omega} |
\partial^{\alpha} u| $$ As usual, $C^{\infty}_c(\Omega)$ is the
intersection of all the $C^k_c(\Omega)$ and is endowed with the $C^{\infty}$-topology. We denote by $supp(f)$ the
support of the function $f$ i.e., the closure of the set $\{x \, | \, f(x) \neq 0 \}$, and by $int(D)$ the interior
of a domain $D \subset {\mathbb R}^n$. For an open domain $ U \subset {\mathbb R}^{2n}$, we denote by
Ham$_c(D,\omega)$ the group of Hamiltonian diffeomorphisms of ${\mathbb R}^{2n}$, which are generated by Hamiltonian
functions $H : {\mathbb R}^{2n} \times [0,1] \rightarrow {\mathbb R}$, whose support is compact and contained in $U
\times [0,1]$. Here $\omega$ is the standard symplectic form on ${\mathbb R}^{2n}$ given by $\omega=dp \wedge dq$,
where $\{q_1,p_1,\ldots,q_n,p_n\}$ are the canonical coordinates in ${\mathbb R}^{2n}$.  We say that a function $f:
{\mathbb R}^{2n} \rightarrow {\mathbb R}$ is a {\it product function}, if it is of the form $f(q,p) = \prod_{i=1}^n
f_i(q_i,p_i)$.
Finally, the letters $C,C_1,C_2,\ldots$ are used to denote positive
constants that depend solely on the dimension of the ambient space
relevant in each particular context.

\noindent {\bf Acknowledgements:} Both authors are grateful to H. Hofer 
and L. Polterovich,
for their interest in this work and helpful comments. This article was written during visits of the
first author at the Institute for Advanced Study (IAS) in Princeton, and visits of the second author at the Mathematical Sciences Research Institute (MSRI), Berkeley. We thank these institutions for their stimulating working atmospheres and for financial support. The first author was supported by the Mathematical Sciences Research Institute. The second author was  supported by NSF Grant DMS-0635607, and by the Israel Science Foundation grant No. 1057/10.
Any opinions, findings and conclusions or recommendations expressed in this material are those of the authors and do not necessarily reflect the views of the NSF or the ISF.

\section{Outline of the Proof} \label{section-outline}

Here we briefly describe the strategy of the proof of
Theorem~\ref{Main-thm}. For technical reasons, we shall prove
Theorem~\ref{Main-thm} for norms on the space $C^{\infty}(M)$,
 instead of the space ${\mathcal A}$.
The original claim would follow from this result since any
Ham$(M,\omega)$ invariant pseudo-norm $\| \cdot \|$ on ${\mathcal
A}$ can be naturally extended to an invariant pseudo-norm $\| \cdot
\|'$ on $C^{\infty}(M)$ by setting $$\| f \|' = \| f- M_f \|, \ {\rm
where \ } M_f = {\textstyle {\frac 1 {Vol(M)}} \int_M f \omega^n}$$
Note that if $\| \cdot \|$ is continuous in the
$C^{\infty}$-topology, then so is $\| \cdot \|'$. Moreover, the norm
$\| \cdot \|'$ coincides with $\| \cdot \|$ on the space ${\mathcal
A}$.
By a standard partition of unity argument, we reduce the proof of
the theorem to a ``local result", i.e., we show that it is
sufficient to prove Theorem~\ref{Main-thm} for
Ham$_c(W,\omega)$-invariant norms on $C_c^{\infty}(W)$, where
$W=(-L,L)^{2n}$ is a $2n$-dimensional cube in ${\mathbb R}^{2n}$.
As a first step toward this end, 
we introduce a special Ham$_c(W,\omega)$-invariant norm $\| \cdot
\|_{{\mathcal F}, max}$ on $C_c^{\infty}(W)$, which depends on a
given finite collection ${\mathcal F} \subset C_c^{\infty}(W)$. More
precisely:
\begin{mydef1} For a non-empty finite collection ${\mathcal F} \subset C^{\infty}_c(W)$, let
$$  \LL_{\mathcal F} := \Bigl \{ \sum_{i} c_{i} \, \Phi_{i}^* {f}_i \ | \ c_i \in {\mathbb R},
\ \Phi_i \in {\rm Ham}_c(W,\omega), \ {f}_i \in {\mathcal F},\  {\rm
and} \ \# \{i \, | \, c_i \neq 0 \} < \infty \Bigr  \}, $$ be
equipped  with the norm $$ \| f \|_{\LL_{{\mathcal F}}} = \inf \sum
|c_i|,$$ where the infimum is taken over all the representations $f
= \sum c_i \, \Phi_i^* {f}_i$ as above.
\end{mydef1}
\begin{mydef2} 
For any compactly supported function $ f \in C_c^{\infty}(W) $, let
\begin{equation*} 
 \| f \|_{{\cal F}, \, max} = \inf \big\{ \liminf_{i \rightarrow \infty} \| f_i
\|_{\LL_{{\mathcal F}}} \big\} ,\end{equation*} where the infimum is
taken over all subsequences $\{f_i\}$ in $ \LL_{\mathcal F} $ which
converge to $f$ in the $C^{\infty}$-topology. As usual, the infimum
of the empty set is set to be $+ \infty$.
\end{mydef2}
The main feature of the norm $\| \cdot \|_{{\cal F}, \, max}$ is 
that it dominates from above any other Ham$_c(W,\omega)$-invariant
norm that is continuous in the
$C^{\infty}$-topology  
(see Lemma~\ref{lemma-about-max-norm}). The next step, which is also the main part of the proof, is to show that for
a suitable collection of functions ${\mathcal F} \subset C_c^{\infty}(W)$, the norm $\| \cdot \|_{{\mathcal F}, \,
max}$ is in turn dominated from above by the $L_{\infty}$-norm. This is proved in Theorem~\ref{Main-Thm-local-case},
and in light of the above, it completes the proof of Theorem~\ref{Main-thm}. The proof of
Theorem~\ref{Main-Thm-local-case} is divided into two main steps which we now turn to describe:

\noindent  {\bf The local two-dimensional case:} Here, we shall
construct a collection ${\mathcal F}$ of smooth compactly supported
functions on a two-dimensional cube $W^2 \subset {\mathbb R}^{2n}$,
such that any  $f \in C_c^{\infty}(W^2)$ satisfies $\| f
\|_{{\mathcal F}, \, max} \leqslant C \| f \|_{{\infty}}$ for some
absolute constant $C$. There are two independent components  in the
proof of this claim. First, we show that one can
 decompose any $f \in C_c^{\infty}(W^2)$ with $\| f \|_{{\infty}} \leqslant 1$
 into a finite combination $f = \sum_{i=1}^{N_0} \epsilon_j \Psi^*_j g_j$. Here, $ \epsilon_j \in \{ -1, 1 \} $, $\Psi_j \in {\rm Ham}_c(W^2,\omega)$, and $g_j$ are smooth radial 
functions  whose $L_{\infty}$-norm is bounded by an absolute
constant, and
%
which satisfy certain other technical conditions (see
Proposition~\ref{prop-about-decompsition2} for the precise
statement). In what follows we call such functions by ``simple functions". We emphasize that $N_0$ is a constant independent of
$f$. Thus, we can restrict ourselves to the case where $f$ is a
``simple function''. In the second part of the proof, we construct
an explicit collection ${\mathcal F} = \{ {\mathfrak f_0},
{\mathfrak f_1} , {\mathfrak f_2} \} $, where ${\mathfrak f_i} \in
C_c^{\infty}(W^2),$ and $i=0,1,2$. 
Using an averaging procedure
(Proposition~\ref{Proposition-about-averaging}), we show that every
``simple function'' $f \in C_c^{\infty}(W^2)$ can be approximated
arbitrarily well in the $C^{\infty}$-topology by a sum of the form
$$ \sum_{i,k} \alpha_{i,k} \widetilde \Psi_{i,k}^* {\mathfrak f}_{k}, \ {\rm where \ }  \widetilde \Psi_{i,k} \in {\rm Ham}_c(W^2,\omega), \ k \in \{0,1,2 \}, $$ 
and such that $\sum | \alpha_{i,k} | \leq C \| f \|_{\infty}$ for
some absolute constant $C$. Combining this with the above definiton
of $\| \cdot \|_{{\mathcal F}, \, max}$, we conclude that
$ \| f \|_{{\mathcal F}, \, max} \leq C \| f \|_{\infty}$, for every
$f \in C_c^{\infty}(W^2)$. This completes the proof of
Theorem~\ref{Main-Thm-local-case} in the 2-dimensional case.

\noindent  {\bf The local higher-dimensional case:} The proof of
Theorem~\ref{Main-Thm-local-case} for arbitrary dimension strongly
relies on the 2-dimensional case. We extend (in a natural way) the
construction of the above mentioned collection ${\mathcal F} = \{
{\mathfrak f_0},{\mathfrak
f_1},{\mathfrak f_2} \}$  
to the $2n$-dimensional case. By abuse of notation, we shall denote
the new collection by ${\mathcal F}$ as well. Based on the proof of
Theorem~\ref{Main-Thm-local-case} in the $2$-dimensional case, and
on the construction of the class ${\mathcal F}$, we show that
Theorem~\ref{Main-Thm-local-case} holds for ``product functions'',
i.e., for $f \in C_c^{\infty}(W)$ of the form $f = \prod_{i=1}^n
f_i(q_i,p_i)$, where $f_i \in C_c^{\infty}(W^2)$.
From this we derive, using a Fourier series argument, that the norm
$\| \cdot \|_{{\mathcal F}, max}$ is dominated from above by the $\|
\cdot \|_{C^{2n+1}}$-norm, i.e., for any $f \in C_c^{\infty}(W)$ one
has \begin{equation} \label{eq-in-outline-about-C2n-norm}  \|f
\|_{{\mathcal F}, \, max} \leq C \|f \|_{C^{2n+1}}, \end{equation}
for some constant $C$  (see Proposition~\ref{Ck-bound-lemma} for the
proof of the above two claims).
Next, for any $\epsilon > 0$, we construct a partition of unity
function ${\mathcal R}^{\epsilon}: \mathbb{R}^{2n} \rightarrow
\mathbb{R} $,
with $ supp( {\mathcal R}^{\epsilon} ) \subset (-\epsilon,\epsilon)^{2n} $, and such that 
$$ \sum_{v \in \epsilon \mathbb{Z}^{2n}}  {\mathcal R}^{\epsilon}(x-v) = \Id(x) $$
For any $w \in { \mathfrak X}:= \{ 0,1,2,3 \}^{2n}$, we consider a finite grid  $\Gamma^{\epsilon}_{w} \subset W$ given by: 
$$ \Gamma^{\epsilon}_{w} =
\epsilon w + 4 \epsilon \mathbb{Z}^{2n} \cap (-L+
3\epsilon,L-3\epsilon)^{2n}, $$ 
and define
$$ f_{w}(x) = \sum_{v \in \Gamma^{\epsilon}_w} {\cal
R}^{\epsilon}(x-v) f(x) $$
Note, that for $ \epsilon $ sufficiently small 
such that $ supp \, (f) \subset (-L+4\epsilon,L-4\epsilon)^{2n} $,
one has $$ f(x) = \sum_{w \in {\mathfrak X}} f_{w}(x) $$ For any $ w
\in {\mathfrak X} $, the function $ f_{w} $ is a finite
sum of smooth functions 
that lie near the points of the grid $ \Gamma^{\epsilon}_{w}$.
Moreover, these functions have mutually disjoint supports, which are
spaced commodiously. 
Next, we fix $ w \in {\mathfrak X} $, and for any
$ v \in \Gamma^{\epsilon}_w $ we consider the decomposition of $f
\in C^{\infty}_c(W)$ as a Taylor polynomial of order $2n+1$ and a
remainder, around the point $ v $ (this specific choice of the order ensure,
based on~$(\ref{eq-in-outline-about-C2n-norm})$, the  estimate~$(\ref{eq-bound-for-h_w})$ below):
 $$ f(x) = P_{2n+1}^v(x-v) +
R_{2n+1}^v(x-v).$$ We decompose each $f_w$ as $ f_w(x) = g_w(x) +
h_w(x) $, where
$$ g_{w}(x) = \sum_{v \in \Gamma^{\epsilon}_w} {\cal
R}^{\epsilon}(x-v) P_{2n+1}^v(x-v) , \ {\rm and}  \ \ h_{w}(x) = \sum_{v \in \Gamma^{\epsilon}_w} {\cal
R}^{\epsilon}(x-v) R_{2n+1}^v(x-v) .$$ Based on~($\ref{eq-in-outline-about-C2n-norm}$), in
Lemma~\ref{lemma-C^k-estimate-of-the-reminder} (cf. Corrolary~\ref{cor-about-max-norm-of-the-reminder}) we show that
the $\| \cdot \|_{{\mathcal F}, max}$-norm of the reminder parts $\{ h_w \}$ can be taken to be arbitrarily small.
More precisely,
\begin{equation} \label{eq-bound-for-h_w}
 \| h_{w} \|_{{\mathcal F}, max} \leqslant C_1 \| h_{w} \|_{C^{2n+1}} \leqslant C_2 \epsilon
\|f\|_{C^{2n+2}} ,
\end{equation}
 for some constants $C_1$ and $C_2$.
On the other hand, using a combinatorial argument and the above
mentioned fact that Theorem~\ref{Main-Thm-local-case} holds for
``product functions", we prove the estimate
\begin{equation} \label{eq-bound-for-g_w}
 \| g_{w} \|_{{\mathcal F}, \, max} \leqslant C_3 \bigl ( \sum_{i=0}^{2n+1}
\|f\|_{C^i} \epsilon^i \bigr)
\end{equation}
for some  constnat $C_3$. Combining the above
estimates~$(\ref{eq-bound-for-h_w})$ and~$(\ref{eq-bound-for-g_w})$
for all $ w \in {\mathfrak X} $, and taking $ \epsilon \rightarrow 0
$, we conclude that for every $f \in C_c^{\infty}(W)$ one has $$ \|
f \|_{{\mathcal F}, \, max} \leqslant C_4 \| f \|_{\infty} ,$$ for
some absolute constant $C_4$. This completes the proof of the
theorem.

\section{A Local Version of the Main Result} \label{section-local-version}

In this section we prove a local version of our main result
(Theorem~\ref{Main-Thm-local-case} below), which would later serve
as the main component in the proof of Theorem~\ref{Main-thm}.

Consider an open cube $ W = I^{2n} \subset \mathbb{R}^{2n} $, where
$ I = (-L,L) \subset \mathbb{R} $ is an open interval. Endow $W$
with linear coordinates $(q_1,p_1,\ldots,q_n,p_n)$, and with the
standard symplectic structure $\omega = dp \wedge dq$ descending
from $
\mathbb{R}^{2n} $. For a finite non-empty collection ${\mathcal F}$ of functions 
in $C_c^{\infty}(W)$, we define the space
$$  \LL_{\mathcal F} := \Bigl \{ \sum_{i} c_{i} \, \Phi_{i}^* {f}_i \ | \ c_i \in {\mathbb R},
\ \Phi_i \in {\rm Ham}_c(W,\omega), \ {f}_i \in {\mathcal F},\  {\rm
and} \ \# \{i \, | \, c_i \neq 0 \} < \infty \Bigr  \} $$ We equip
${\mathcal L}_{\mathcal F}$ with the norm
$$ \| f \|_{\LL_{{\mathcal F}}} := \inf   \sum |c_i|, $$
where the infimum is taken over all the representations 
$f = \sum c_i \, \Phi_i^* {f}_i$ as above.
\begin{definition} \label{definition1-of-our-max-norm-local}
For any compactly supported function $ f \in C_c^{\infty}(W) $, let
\begin{equation} \label{definition-of-max-norm-local}
 \| f \|_{{\cal F}, \, max} = \inf \big\{ \liminf_{i \rightarrow \infty} \| f_i
\|_{\LL_{{\mathcal F}}} \big\} ,\end{equation} where the infimum is taken over all subsequences $\{f_i\}$ in $
\LL_{\mathcal F} $ which converge to $f$ in the $C^{\infty}$-topology. If such sequence do not exists, we set $ \| f
\|_{{\cal F}, \, max} \equiv +\infty$.
\end{definition}

\begin{remark} \label{rmk-about-max-norm} {\rm It follows from the definition above
that $\| \cdot \|_{{\cal F}, \,max}$ is homogeneous, Ham$_c(W,\omega)$-invariant, and satisfies the triangle
inequality\footnote{When $\| \cdot \|_{{\cal F}, \,max} \equiv + \infty $, these statements are trivially true. }.
Moreover, let $ \{ f_k  \}$ be a sequence of smooth functions that converge in the $C^{\infty}$-topology to $f$, and
such that for every $k \geqslant 1$ one has $ \|f_k\|_{{\cal F}, \,max} \leqslant C $ for some constant $C$. Then $
\|f\|_{{\cal F}, \,max} \leqslant C $. The fact that $\| \cdot \|_{{\cal F}, \,max}$ is non-degenerate (i.e., $\| f
\|_{{\cal F}, \,max} = 0$ if and only if $f=0$) follows from the next lemma.
}
\end{remark}
\begin{lemma} \label{lemma-about-max-norm}
Let ${\mathcal F} \subset C_c^{\infty}(W)$ be a non-empty finite collection of smooth compactly supported functions
in $W$. Then, any Ham$_c(W,\omega)$-invariant norm $ \| \cdot \| $ on $ C_c^\infty(W) $ which is continuous in the $
C^\infty $-topology, satisfies $ \| \cdot \| \leqslant C \| \cdot \|_{{\cal F}, \,max}$ for some absolute constant
$C$.
\end{lemma}
\begin{proof}[\bf Proof of Lemma~\ref{lemma-about-max-norm}]
Let $C = \max \{ \| g \|  ; \, g \in {\mathcal F} \}$. For any $f = \sum c_i \, \Phi_i^*  f_i \in {\cal L}_{\mathcal
F}$, one has:
\begin{equation} \label{simpel-estimate1}  \|f \| \leq   \sum |c_i| \| \Phi_i^* f_i \|  \leq  C \sum |c_i| \leq C  \| f \|_{{\cal F}, \,max} \end{equation}
The lemma now follows from combining~$(\ref{simpel-estimate1})$, definition~$(\ref{definition-of-max-norm-local})$,
and the fact that the norm $ \| \cdot \| $ is assumed to be continuous in the $ C^\infty $-topology.
\end{proof}

The following theorem, which is a ``local version" of
Theorem~\ref{Main-thm}, shows that for a suitable choice of a
collection ${\mathcal F}$, the subspace $\LL_{{\mathcal F}} \subset
C_c^{\infty}(W)$ is dense in the $C^{\infty}$-topology,
 and moreover, that the norm $\| \cdot \|_{{\cal F}, \, max}$ on $C^{\infty}_c(W)$ is dominated from above by the $\| \cdot \|_{\infty}$-norm.

\begin{theorem} \label{Main-Thm-local-case}
There is a finite collection ${\mathcal F} \subset 
C_c^{\infty}(W)$, such that $ \| \cdot \|_{{\cal F}, \, max} $ is a
genuine norm on
$C_c^{\infty}(W)$, and 
$\| \cdot \|_{{\mathcal F}, \, max} \leq C \, \|\cdot \|_{\infty}$
for some absolute constant $C$.
\end{theorem}

The remainder of this section is devoted to the proof
of Theorem~\ref{Main-Thm-local-case}, which we split 
into two separate cases: 

\subsection{Theorem~\ref{Main-Thm-local-case} - the two-dimensional case} \label{subsection-thm3.4-the-2-dim-case}

We assume that $n=1$, and hence $W = (-L,L) \times (-L,L)$. We set
$z = x + iy$, where $\{x,y\}$ are local coordinates on $W$, and
denote by $D_a = \{ |z| \leq a \}$ the disc with radius $a$ centered
at the origin, and by $D_{a,A} = \{ a \leq |z| \leq A\} $ the
annulus with radii $a,A$. The proof of
Theorem~\ref{Main-Thm-local-case} in the two-dimensional case
follows from the next two propositions, the proof of which we
postpone to Subsections~\ref{subsection-prop-about-decompsition2}
and~\ref{subsection-Proposition-about-averaging}.

\begin{proposition} \label{prop-about-decompsition2}
There are positive constants $a,A,C$ such that $a < A < L$; a smooth
radial function ${\mathfrak f_1}$ with $supp({\mathfrak f_1}) = D_A$; 
and an integer number $N_0 \in {\mathbb N}$, such that every $f \in
C_c^{\infty}(W)$ with $\|f \|_{\infty} \leqslant 1$ can be
decomposed as
$$f =   \sum_{j=1}^{N_0}  \epsilon_j \, \Phi_j^* g_j,$$ 
where $\Phi_{j} \in Ham_c(W,\omega)$, $ \epsilon_j \in \{ -1,1 \} $, and $g_j$ are smooth radial
functions that satisfy: 
\begin{equation} \label{eq-properties-of-g_j} \ supp(g_j) = D_A, \  
g_j \equiv {\mathfrak f_1} \ {\rm  on} \ D_a, \ {\rm and} \ \|g_j\|_{\infty} \leqslant C \end{equation} 
\end{proposition}

\begin{proposition} \label{Proposition-about-averaging}
Let $ 0 < a < A $ be positive numbers. Then there exists a smooth
function $F_{a,A} : \mathbb{R}^2 \rightarrow \mathbb{R} $ with
 $ supp(F_{a,A}) \subset D_A$,  such that the following holds:
for every smooth radial function $ f : \mathbb{R}^2 \rightarrow
\mathbb{R} $, that satisfies
\begin{equation} \label{eq-with-cond-on-the-func-in-the-ave-argument}  \|f \|_{\infty} \leqslant 1,  \ supp(f) \subset D_{a,A}, \ {\rm and}  \ \int_{\mathbb{R}^2} f \omega = 0, \end{equation}
%
%
there exists an area-preserving diffeomorphism $ \Phi : \mathbb{R}^2 \rightarrow \mathbb{R}^2 $, with $ supp(\Phi)
\subset D_{a,A}$, 
and such that: 
$$ \int_{D_r} \Phi^* F_{a,A} \,
\omega = \int_{D_r} f  \omega ,  \ \ {\rm for \ any \ } r>0$$
\end{proposition}

We are now in a position to prove Theorem~\ref{Main-Thm-local-case}
in the two-dimensional case.
\begin{proof}[{\bf Proof of Theorem~\ref{Main-Thm-local-case}}  {(the 2-dimensional case):}] Let $f \in C_c^{\infty}(W)$ with $ \| f\|_{\infty} \leqslant 1 $.
It follows from Proposition~\ref{prop-about-decompsition2} above
that there are 
positive constants $a,A,C$, an integer $N_0$, and a smooth radial
function ${\mathfrak f_1}$ with $supp({\mathfrak f_1}) = D_A$,
such that $f$ 
can be written as
$$f =   \sum_{j=1}^{N_0} \epsilon_j \, \Phi_j^* g_j,$$ 
where $\Phi_{j} \in {\rm Ham}_c(W,\omega)$, $\epsilon_j \in \{ -1,1 \}$, and $\{g_j\}$ are smooth radial functions
that satisfy~$(\ref{eq-properties-of-g_j})$.
%
Next, let ${\mathfrak f_2}$ be a smooth radial function with $supp
({\mathfrak f_2}) = D_{a,A}$ such that $\int_{W^2} {\mathfrak f_2}
\, \omega = 1$. Moreover, let ${\mathfrak f_0} = F_{a,A}$ be the
function provided by Proposition~\ref{Proposition-about-averaging}
above.
We consider the function $$h_j := g_j - {\mathfrak f_1} - c_j
{\mathfrak f_2}, \ {\rm where} \  c_j = \int_W (g_j - {\mathfrak
f_1}) \, \omega$$ Note that there exists a constant $C'$ such that
$\| h_j \|_{\infty} \leq C'$. Indeed:
$$ \| h_j \|_{\infty} \leq C + \| {\mathfrak f_1} \|_{\infty} 
+ |c_j| \| {\mathfrak f_2} \|_{\infty} \leq C + \| {\mathfrak f_1} \|_{\infty} 
+  \| {\mathfrak f_2} \|_{\infty} \, \Bigl ( \pi \, C A^2 + \int_{W}
| {\mathfrak f_1}| \, \omega \Bigr )  $$
From Proposition~\ref{Proposition-about-averaging} it follows that
there are area-preserving diffeomorphisms
$\widetilde \Phi_j$ with 
 $ supp(\widetilde \Phi_j) \subset D_{a,A}$, 
 such that for any $ r > 0 $ one has \begin{equation} \label{equation-int-are-eq-1} \int_{D_r} (\widetilde \Phi_j^* {\mathfrak f_0} ) \, \omega = {\frac 1 {C'} } \int_{D_r}  h_j \, \omega
 \end{equation}

To complete the proof of the theorem, we shall need the following
technical lemma:
\begin{lemma} \label{tech-lemma-about-aver-that-converge-nicely}
Let $f \in C_c^{\infty}(D)$ be a compactly supported function in a
disk $D$. Then
$${\frac 1 N}
\sum_{i=1}^N f(ze^{\frac {2 \pi i} N})
\xrightarrow[]{N \rightarrow \infty } 
{\frac 1 {2 \pi}} \int_0^{2 \pi} f(ze^{i \theta}) \, d \theta, \ \
{\it in \ the \ } C^{\infty} \, {\it topology}$$
\end{lemma}
Postponing the proof of
Lemma~\ref{tech-lemma-about-aver-that-converge-nicely}, we first
finish the proof of the theorem.

Consider a compactly supported Hamiltonian isotopy $T^A_{\theta} : W
\rightarrow W$, where $\theta \in {\mathbb R}$, and
such that $T^A_{\theta}(z) =  e^{i \theta} z $ in $D_A$. 
From Lemma~\ref{tech-lemma-about-aver-that-converge-nicely}
and~$(\ref{equation-int-are-eq-1})$ it follows that:
\begin{equation} \label {eq-approx-hj-by-average} {\frac {C'} N} \sum_{k=1}^N (T^A_{{\frac {2 \pi k}
{N}}})^* \widetilde \Phi_j^* {\mathfrak f_0} \xrightarrow[]{N
\rightarrow \infty } h_j, \ \ {\it in \ the \ } C^{\infty} \ {\it
topology} \end{equation} We set ${\mathcal F} = \{ {\mathfrak
f_0},{\mathfrak f_1},{\mathfrak f_2} \}$. 
From~$(\ref{eq-approx-hj-by-average})$ and
Remark~\ref{rmk-about-max-norm} it follows that
$ \|  h_j \|_{{\mathcal F},max} \leq C' $. Moreover, by definition one has:
 $\| {\mathfrak f_1} \|_{{\mathcal F},max}, \| {\mathfrak f_2}
\|_{{\mathcal F},max} \leq 1$.
 This implies that
 $$ \|
g_j \|_{{\mathcal F},max} \leq C'',$$ 
where $C''$ is an absolute constant given by: $$C'' = C' + 1 + \pi C A^2 + \int_W | {\mathfrak f_1 }| \, \omega$$
 Thus, we conclude that $ \| f   \|_{{\mathcal F},max} \leq N_0 \, C'' $. This completes the proof of the theorem.
\end{proof}

\begin{proof}[{\bf Proof of Lemma~\ref{tech-lemma-about-aver-that-converge-nicely}}]
We shall prove the convergence
$${\frac 1 N}
\sum_{i=1}^N f(ze^{\frac {2 \pi i} N}) \xrightarrow[]{N \rightarrow
\infty } {\frac 1 {2 \pi}} \int_0^{2 \pi} f(ze^{i \theta}) \, d
\theta $$ in $ C_c^k(D) $, for any $ k \in \mathbb{N} $. Note that
the operators $ P_N(f) = {\frac 1 N} \sum_{i=1}^N f(ze^{\frac {2 \pi
i} N}) $, defined on the space $C_c^k(D) $, have a bounded operator
norm which is independent on $ N $. Therefore, it is enough to check
that
$$ P_N f \xrightarrow[]{N \rightarrow \infty } {\frac 1 {2 \pi}} \int_0^{2 \pi} f(ze^{i \theta}) \, d
\theta ,$$ in $C_c^k(D)$ only on some dense subspace. 
We choose this subspace to be consists of all the finite sums:
$$ s_m(z) = \sum_{l=0}^{m} u_l(r) \cos(l\theta) + v_l(r)
\sin(l\theta),$$ where $u_l$ and $v_l$ are smooth radial functions
supported in the disk $D$.
Note that for $ N > m $ one has $$ P_N s_m(z) = u_0(r) = {\frac 1 {2
\pi}} \int_0^{2 \pi} s_m(ze^{i \theta}) \, d \theta ,$$ and hence
the statement of the lemma is satisfied 
in a trivial way. The proof of the lemma is now complete.
\end{proof}
We now return to complete the proof of
Proposition~\ref{prop-about-decompsition2} and
Proposition~\ref{Proposition-about-averaging}.
\subsubsection{Proof of Proposition~\ref{prop-about-decompsition2}}
\label{subsection-prop-about-decompsition2} For the sake of clarity,
we fragment the proof of the proposition in several steps:

\noindent {\bf Step I:} We choose $ a = \frac{L}{4} $, $ A =
\frac{L}{2} $. The area of the sector $$ \{ z \in W \ | \ a < |z| <
A \ ; \ 0 < Arg \, z < \frac{\pi}{2} \} $$ equals to $$
\frac{\pi}{4} \left( \frac{L^2}{4}
- \frac{L^2}{16} \right) = \frac{3\pi}{64}L^2 > 
\frac{L^2}{8} = \frac{Area(W)}{32}$$ Using a smooth partition of
unity, one can decompose $f$ as $ f = \sum_{k=1}^{33} f_k $, where
the support of each $ f_k $ lies in an open sub-rectangle of the
square $ W $ of area $ \frac{Area(W)}{32} $, and $ \| f_k
\|_{\infty} \leq 1 $. Next, we take compactly supported
area-preserving diffeomorphisms 
$ \widetilde{\Phi}_k : W \rightarrow W $, such that $ f_k =
\widetilde{\Phi}_k^* f_k',$ for $ k=1,\ldots,33 $, and $ supp (f_k')
\subset (0,\frac{L}{4}) \times (0,\frac{L}{2}) $. Denote $ L_1 =
\frac{L}{4} $ and $L_2 = \frac{L}{2} $. From the above we conclude
that it is enough to restrict ourselves to the case where $ supp (f)
\subset (0,L_1) \times (0,L_2) $. Indeed, if the proposition holds
for such functions, then by replacing $ N_0 $ with $ 33N_0 $, it
will hold for any compactly supported function $f \in
C_c^{\infty}(W)$. 

\noindent {\bf Step II:} Following Step I, we assume that $ supp (f)
\subset (0,L_1) \times (0,L_2) $. Next, we apply the following lemma to the function $f$.

\begin{lemma} \label{lemma-about-decomposition}
Let $ R = [0,L_1] \times [0,L_2] \subset \mathbb{R}^2 $ be a
rectangle, and let $ f: \mathbb{R}^2 \rightarrow \mathbb{R} $ be a smooth function with $ supp(f) \subset int(R)$, 
and $ \| f \|_{\infty} \leqslant 1 $. Then there exists a
decomposition $ f = \sum_{i=1}^{8} f_i $, and compactly supported
diffeomorphisms $ \Psi_i : R \rightarrow R $, $ i=1,2,...,8 $, such
that the functions $ g_i: = \Psi_i^* f_i $ satisfy $ |
\frac{\partial}{\partial x} g_i | \leqslant \frac{12}{L_1} $.
\end{lemma}
The proof of Lemma~\ref{lemma-about-decomposition} will be given in Subsection~\ref{subsection-techniacal-lemmata}.
\begin{remark} \label{rmk-about-a-reduction-step-for-f}
{\rm Analogously to Step I,
Lemma~\ref{lemma-about-decomposition} reduces the proposition to the
case where $ supp (f) \subset (0,L_1) \times (0,L_2) $, and moreover
that there is a diffeomorphism $ \Psi :  W \rightarrow W $ with $
supp(\Psi) \subset (0,L_1) \times
(0,L_2) $, such that 
$ g = \Psi^* f $ satisfies $ |\frac{\partial}{\partial x} g |
\leqslant \frac{12}{L_1} $. Indeed, the general case would follow by
replacing $ N_0 $ with $ 8 \cdot 33
\cdot N_0 = 264 N_0$. Thus, we assume in what follows the existence of 
$ f,g$ and $\Psi $ as above.}
\end{remark}

\noindent {\bf Step III:} Denote by ${\mathcal R}$ the rectangle
$[0,L_1] \times
[0,L_2]$.  From the fact that $$ Area({\mathcal R}) 
 < Area( \{ z \in W \ | \ a < |z| < A \ ; \ 0 < Arg \, z <
\frac{\pi}{2} \}) ,$$ one can easily find an area preserving
diffeomorphism $ \Phi: W \rightarrow W $ with
$$ \Phi({\mathcal R}) 
= \{ z \in W \ | \ a < |z| < A_1 \ ; \ 0 < Arg \, z <\frac{\pi}{2}
\}, $$
for an appropriate $ a < A_1 < A $; and such that on ${\mathcal R}$, the diffeomorphism 
$ \Phi $ takes the form $ \Phi(x+iy) = r_1(x)e^{\theta_1 (y)} $, where $r_1(x)$ is a monotone increasing function. Let
$ C_1 = \min_{x \in [0,L_1]} r_1'(x) >
0, $ and define $ h = (\Phi^{-1})^* g $. Note that one can bound the radial derivative of $h$ by: 
$$ \max
|\frac{\partial}{\partial r} h| \leq \frac{1}{C_1} \max
|\frac{\partial}{\partial x} g| \leq \frac{12}{L_1 C_1}$$
Next, we set $C_2= \frac{12}{L_1 C_1}$, and fix a smooth radial
function $ { \mathfrak f_1 } $ such that
$$ supp ({\mathfrak f_1 }) \subset D_A, \ \  {\textstyle  \frac{\partial}{\partial r}} {
\mathfrak f_1} (z) < -C_2 \ {\rm for} \ z \in D_{a,A_1}, \ \
{\textstyle \frac{\partial}{\partial r}} { \mathfrak f_1 } (z)  < 0
\ {\rm for} \ z \in int(D_A) \setminus \{ 0 \},$$
and such that the point $z=0 $ is a non-degenerate maximum for the
function $ { \mathfrak f_1 } $. We denote $ H = h + { \mathfrak f_1
} (z)  $, and observe that $H$ satisfies:
$$ supp(H) \subset D_A, \ \ {\textstyle \frac{\partial}{\partial r}}  H  < 0 \ {\rm in \ } int(D_{A}) \setminus \{ 0 \},  \ \
 H(z) \equiv { \mathfrak f_1} (z)  \ {\rm in \ } D_a \cup D_{A_1,A},
$$
and that the point $ z = 0 $ is a unique non-degenerate critical
point of $ H $,
which is a maximum point.
Consider the gradient flow of $H$. By a standard Morse theory
argument one can find a diffeomorphism $ \Upsilon : W \rightarrow W
$, with $ supp (\Upsilon) \subset D_{a,A} $, and such that $ K:=
\Upsilon^* H $ is a radial function. Finally, we have
\begin{eqnarray*} f & = &
(\Psi^{-1})^*g = (\Psi^{-1})^* \Phi^{*} h = (\Psi^{-1})^* \Phi^{*} H -
(\Psi^{-1})^* \Phi^{*} { \mathfrak f_1 }  \\ & = & 
(\Psi^{-1})^* \Phi^{*}
(\Upsilon^{-1})^* K - (\Psi^{-1})^* \Phi^{*} { \mathfrak f_1 } . \end{eqnarray*} Note,
that for $ z \in W \setminus D_{a,A} $, one has
$$ \Psi \Phi^{-1} \Upsilon (z)  =  \Psi \Phi^{-1} (z)
= \Phi^{-1} (z) $$ Indeed, this follows from the fact that $
supp(\Psi) \subset {\mathcal R} \subset  \Phi^{-1} (D_{a,A})$, and
that $\Upsilon $ is the identity on the complement $ W \setminus
D_{a,A} $.
Thus, we conclude that
$$ (\Psi \Phi^{-1} \Upsilon)^* \omega = (\Psi \Phi^{-1} )^* \omega = \omega, \ {\rm  on \ the \ complement \ } W \setminus D_{a,A} $$

Next, let $S_r =  \{ z \in W \, | \, |z|=r \} $. We shall need the
following lemma:
\begin{lemma} \label{lemma-reparametrization}
Let $ \omega' $ be a symplectic form on $ W$ which coincides with
the standard symplectic form $\omega$ on the complement $W \setminus
D_{a,A} $, and such that
$ \int_{W} \omega' = \int_{W} \omega $. Then, there exists a
diffeomorphism $ \Lambda : W \rightarrow W $ supported in $D_{a,A}$,
such that for every $a < r < A$, one has $\Lambda(S_r)=S_R$, for
some $a < R < A$, and such that $ \Lambda^* \omega = \omega' $.
\end{lemma}
\begin{proof}[{\bf Proof of Lemma~\ref{lemma-reparametrization}}]
Consider the function $ S:[0,L) \rightarrow [0, \infty)$, defined by
$ S(r) = \int_{D_r} \omega' $. Note that $S$ is a smooth function,
and that $ S(r) = \pi
r^2 $ for every $ r \in [0,a]\cup [A,L)$. 
Define a diffeomorphism $ \Delta_1 : W \rightarrow W $, supported in
$ D_{a,A}$, by $$ \Delta_1(r,\theta) = \left(
\sqrt{\frac{S(r)}{\pi}}, \theta \right), \ \ {\rm for} \ r \in [0,L),$$  and
extend it by the identity diffeomorphism to the whole $ W $. Denote
$ \omega'' = (\Delta_1^{-1})^* \omega' $, and note that 
$ \int_{D_r} \omega'' = \pi r^2 $ for $ r \leqslant A $, and $
\omega'' = \omega'=\omega $ on $ W \setminus D_{a,A} $. Next, we
explicitly construct a diffeomorphism $ \Delta_2 : W \rightarrow W $
supported in $
D_{a,A} $, such that $ \omega'' = \Delta_2^* \omega $, and 
 for $ 0 < r
< L $, it takes the form $ \Delta_2(r,\theta) = (r,F(r,\theta)) $,
for some smooth map $ F: (0,L) \times S^1 \rightarrow S^1 $. To this
end, note that  $ \omega'' = G \omega $ for some positive function $
G : W
\rightarrow (0,\infty) $, such that $ G = 1 $ on $ W \setminus D_{a,A} $. Moreover, 
$$ \pi r^2 = \int_{D_r} \omega'' =
\int_{D_r} G\omega  , \ \ {\rm for \ all} \ \ 0 < r < L$$
After differentiating this equality we obtain
\begin{equation} \label{eq-funct-G-integral}
 \int_{0}^{2\pi} G(r,\theta) \, d  \theta = 2\pi , \ \ {\rm for \ every} \ \ 0 < r < L
\end{equation}
On the other hand, we require $\Delta_2$ to satisfy:
$$ \Delta_2^* \omega = r F_{\theta}(r,\theta) dr \wedge d\theta = F_{\theta}(r,\theta) \omega , \ \ {\rm for \ every \ } r \in (0,L)$$
Thus, the condition $ \omega'' = \Delta_2^* \omega $ is equivalent
to $ F_{\theta}(r,\theta) = G(r,\theta) $, for $ r \in (0,L) $. We
define \begin{equation} \label{eq-def-of-the-func-F} F(r,\theta) =
\int_{0}^{\theta}
G(r,s) \, ds , \ \ {\rm for} \ r \in (0,L), \ \theta \in [0,2\pi) \end{equation} 
In light of~$(\ref{eq-funct-G-integral})$, we obtain a smooth map $
F: (0,L) \times S^1 \rightarrow S^1 $. Moreover, since $ G = 1 $ on
$ W \setminus D_{a,A} $, one has $ F(r,\theta) = \theta $ for $
r \in (0,a]\cap[A,L)$. 
Therefore, defining $ \Delta_2(r,\theta) = (r,F(r,\theta)) $ for $ 0
< r < L $, where $F$ is given in~$(\ref{eq-def-of-the-func-F})$,
we obtain a diffeomorphism of $ D_{L} $ supported in $ D_{a,A} $. We
extend $\Delta_2$ to the whole $W$ by the identity diffeomorphism.
Note that $ \omega'' = \Delta_2^* \omega $, and hence $ \omega' =
\Delta_1^* \omega'' = \Delta_1^* \Delta_2^* \omega $. Denoting $
\Lambda = \Delta_2 \Delta_1 $, we conclude the statement of the
lemma.
\end{proof}

We return now to the proof of the Proposition. By applying Lemma~\ref{lemma-reparametrization}
to the forms $ \omega' =  (\Psi
\Phi^{-1} \Upsilon)^* \omega $ and $ \omega'' =  (\Psi \Phi^{-1})^*
\omega $, we
obtain two diffeomorphisms $
\Lambda',\Lambda'' $  
such that $ \Lambda'^{*} \omega = (\Psi \Phi^{-1} \Upsilon)^* \omega
$, and $ \Lambda''^{*} \omega = (\Psi \Phi^{-1})^* \omega $. Denote $
\Phi' := \Lambda' \Upsilon^{-1} \Phi \Psi^{-1} $, $ \Phi'' :=
\Lambda'' \Phi \Psi^{-1} $. Note that $ \Phi', \Phi'' \in
{\rm Ham}_c(W,\omega) $, and that 
 \begin{eqnarray*} f & = & (\Psi^{-1})^* \Phi^{*} (\Upsilon^{-1})^* K - (\Psi^{-1})^* \Phi^{*} { \mathfrak f_1 }  = (\Psi^{-1})^* \Phi^{*} (\Upsilon^{-1})^* (\Lambda')^{*} K - (\Psi^{-1})^* \Phi^{*} (\Lambda'')^{*}{ \mathfrak f_1 } \\ & = & (\Phi')^* K - (\Phi'')^* { \mathfrak f_1 } \end{eqnarray*} The decomposition $  f = (\Phi')^* K - (\Phi'')^* { \mathfrak f_1 } $ shows that the proposition holds for
$f$ as in Remark~\ref{rmk-about-a-reduction-step-for-f}, with only
two summands in the decomposition, and with
$ C = \| { \mathfrak f_1 } \|_{\infty} $. Therefore, we obtain the
conclusion of Proposition~\ref{prop-about-decompsition2} with $ N_0
= 264 \cdot 2 = 528$.

\subsubsection{Proof of Proposition~\ref{Proposition-about-averaging}}
\label{subsection-Proposition-about-averaging}

We start with a construction of a function $F$, such that
for any smooth radial function $ f : \mathbb{R}^2 \rightarrow
\mathbb{R} $, satisfying the conditions~$(\ref{eq-with-cond-on-the-func-in-the-ave-argument})$
one can find a diffeomorphism (not necessarily
area-preserving) $ \Psi : \mathbb{R}^2 \rightarrow \mathbb{R}^2 $
supported in $D_A$
such that for any $ r > 0 $:
\begin{equation} \label{E:condition1-on-Psi}
\int_{D_r} \Psi^*\omega = \int_{D_r} \omega = \pi r^2,
\end{equation}
and,
\begin{equation} \label{E:condition2-on-Psi}
\int_{D_r} \Psi^* (F \omega) = \int_{D_r} f \omega .
\end{equation}

We shall take the function $F$ to be of the form $ F(r,\theta) =
\phi(r)\psi(\theta) $, where $ \phi,\psi $ are smooth functions. We
assume that $ \phi(r) = 0 $, for small enough $r$, and that $ \phi(r) = 1
$ for $ r \geqslant a $. The function $\psi$ is assumed to satisfy
$\int_{0}^{2\pi} \psi(\theta) d\theta = 0$, and would be determined
in the sequel. Moreover,

\begin{equation} \label{conditions-on-R}
 \left\{
\begin{array}{lllll}
R(r,\theta) = \sqrt{ u(r)\mu(\theta) +
v(r) \nu(\theta) },\\[0.05in]
u(r) = v(r) = r^2  \ \ {\rm for}  \ r \leqslant a  \ {\rm or \  }
r \geqslant A, \\[0.05in]
u'(r),v'(r) > 0 \ {\rm for \ }  r > 0 , \\[0.05in]
\mu(\theta),\nu(\theta)
\geqslant 0, \\[0.05in]
\mu(\theta) + \nu(\theta) = 1
\end{array} \right.
\end{equation}
Here, $\mu,\nu,u$ and $v$, are smooth functions that
would be determined
explicitly in the sequel.
Note that conditions~$(\ref{conditions-on-R})$ 
ensure that $\Psi$ is a diffeomorphism of ${\mathbb R}^2$ supported
in $D_{a,A}$. Next, we compute $$ \Psi^*\omega =
R(r,\theta)R'_r(r,\theta) dr \wedge d\theta  = \frac{1}{2} \big(
u'(r)\mu(\theta) + v'(r) \nu(\theta) \big) dr \wedge d\theta , $$
 and
 \begin{eqnarray*} \Psi^*(F\omega) & = & F(R(r,\theta),\theta)R(r,\theta)R'_r(r,\theta) dr \wedge d\theta \\ & = & \frac{1}{2} \phi(R(r,\theta))\psi(\theta) \big( u'(r)\mu(\theta) + v'(r) \nu(\theta) \big) dr \wedge d\theta . \end{eqnarray*}
After differentiating by $r$ and some simplification,
conditions~$(\ref{E:condition1-on-Psi})$,~$(\ref{E:condition2-on-Psi})$
become
\begin{equation} \label{new-cond1-Psi}
u'(r) \int_{0}^{2\pi} \mu(\theta) d\theta + v'(r) \int_{0}^{2\pi}
\nu(\theta) d\theta = 4 \pi r
\end{equation}
and,
\begin{equation} \label{new-cond2-Psi}
u'(r) \int_{0}^{2\pi} \phi(R(r,\theta))\psi(\theta) \mu(\theta)
d\theta + v'(r) \int_{0}^{2\pi} \phi(R(r,\theta)) \psi(\theta)
\nu(\theta) d\theta = 4 \pi r f(r)
\end{equation}
Note that when $ r \geqslant a $, one has $ R(r,\theta) \geqslant a
$, and condition~$(\ref{new-cond2-Psi})$ turns to:
\begin{equation} \label{new-cond2-Psi-r>=a}
u'(r) \int_{0}^{2\pi} \psi(\theta) \mu(\theta) d\theta + v'(r)
\int_{0}^{2\pi} \psi(\theta) \nu(\theta) d\theta = 4 \pi r f(r)
\end{equation}
Next, we choose the functions $ \psi, \mu, \nu $ to be any smooth
functions satisfying:
\begin{equation} 
 \left\{
\begin{array}{lllll}
 \int_{0}^{2\pi} \psi(\theta) \mu(\theta) d\theta = 2\pi,\\[0.07in]
\int_{0}^{2\pi} \psi(\theta) \nu(\theta) d\theta = -2\pi, \\[0.07in]
\int_{0}^{2\pi} \mu(\theta) d\theta = \int_{0}^{2\pi} \nu(\theta) d\theta = \pi, \\[0.07in]
\mu(\theta),\nu(\theta)
\geqslant 0, \\[0.07in]
\mu(\theta) + \nu(\theta) = 1
\end{array} \right.
\end{equation}
Note that this choice of $ \psi, \mu, \nu $ do not depend on the function $ f $.
Moreover, with the above choice, for $ r \geqslant a $,
equations~$(\ref{new-cond1-Psi})$ and~$(\ref{new-cond2-Psi-r>=a})$
become
\begin{equation} \label{equation-about-u-and-v}
 \left\{
\begin{array}{ll}
  u'(r) + v'(r) = 4r,\\[0.07in]
u'(r) - v'(r) = 2 r f(r)
\end{array} \right.
\end{equation}
Next, we consider equations~$(\ref{equation-about-u-and-v}$) for
every $r \geqslant 0 $, with initial conditions $ u(0) = v(0) = 0 $.
There is no difficulty in checking that the solutions of this system
are
\begin{equation} 
 \left\{
\begin{array}{ll}
 u(r) = \int_{0}^{r} s(2+f(s)) ds ,\\[0.07in]
v(r) = \int_{0}^{r} s(2-f(s)) ds
\end{array} \right.
\end{equation}
One can easily check, that as required, the function $u$ and $v$ satisfy
\begin{equation}
 \left\{
\begin{array}{ll}
u'(r),v'(r) > 0, \ \ {\rm for} \ r > 0, \\[0.07in]
u(r) = v(r) = r^2, \ \ {\rm for} \ r
\leqslant a \ \ {\rm and} \ r \geqslant A
\end{array}  \right.
\end{equation}
Moreover, by
definition, they satisfy equations~$(\ref{new-cond1-Psi})$
and~$(\ref{new-cond2-Psi})$ when $ r \geqslant a $. Let us now show
that these equations hold
for $ r < a $ as well. First, note that equation~$(\ref{new-cond1-Psi})$
clearly holds when $ r < a $. Second, by defintion, for $ r < a $ one has $ u(r) =
v(r) = r^2 $, and $ R(r,\theta) = r $.  Hence, we compute
$$ u'(r)
\int_{0}^{2\pi} \phi(R(r,\theta))\psi(\theta) \mu(\theta) d\theta +
v'(r) \int_{0}^{2\pi} \phi(R(r,\theta)) \psi(\theta) \nu(\theta)
d\theta
$$ $$ = u'(r)\phi(r) \int_{0}^{2\pi} \psi(\theta) \mu(\theta)
d\theta + v'(r) \phi(r) \int_{0}^{2\pi} \psi(\theta) \nu(\theta)
d\theta $$
$$ = 2r\phi(r) \Bigl( \int_{0}^{2\pi} \psi(\theta) \mu(\theta) d\theta + \int_{0}^{2\pi} \psi(\theta) \nu(\theta) d\theta \Bigr) $$ $$ = 2r\phi(r)(2\pi-2\pi) = 0$$
Combining this with the fact that $supp(f) \subset D_{a,A}$, we obtain that
%
the functions $u$ and $v$, 
satisfy~$(\ref{new-cond1-Psi})$ and~$(\ref{new-cond2-Psi})$ for all
$ r \geqslant 0 $. We conclude that the resulting diffeomorphism $
\Psi $ satisfies conditions~$(\ref{E:condition1-on-Psi})$
and~$(\ref{E:condition2-on-Psi})$.
Furthermore, since the diffeomorphism $ \Psi $ satisfies~$(\ref{E:condition1-on-Psi})$, and $supp(\Psi) \subset
D_{a,A}$, by using a similar arguments as in the proof of Lemma (3.1.5) from~\cite{Schl}, we conclude that there
exists an area-preserving diffeomorphism $ \Phi : \mathbb{R}^2 \rightarrow \mathbb{R}^2 $,
with $supp(\Phi) \in D_{a,A}$, such that  $ \Phi(D(r)) = \Psi(D(r)) $ for any $ r > 0 $. Thus, we
obtain
$$ \int_{D_r} (\Phi^* F) \omega = \int_{D_r} \Phi^* (F \omega) = \int_{\Phi(D_r)} F \omega = \int_{\Psi(D_r)} F \omega  = \int_{D_r} \Psi^* (F \omega) = \int_{D_r} f \omega,$$
and the proof of the Proposition in now complete.

\subsubsection{Technical Lemmata} \label{subsection-techniacal-lemmata}

In this subsection we prove Lemma~\ref{lemma-about-decomposition}
which was used in the proof of
Proposition~\ref{prop-about-decompsition2}. We start with the
following preparation:

\begin{lemma} \label{L:special-function}
There is a smooth function $ \phi : \mathbb{R} \rightarrow
\mathbb{R} $ with the following properties:
\begin{enumerate}\addtolength{\itemsep}{-1.5\baselineskip}
 \item  $supp (\phi) = [0,3],$ \\[0.1in]
\item  $\phi(t) > 0, \ {\rm for } \ t \in (0,3)$, \\[0.1in]
\item $\phi'(t) > 0, \ {\rm for } \ t \in (0,3/2), \ {\rm and} \ \phi'(t) < 0 \ {\rm for } \ t \in (3/2,3),$ \\[0.1in]
\item $\sup\limits_{t \in (0,3)}  \Bigl ( \frac{\phi'(t)}{\phi(t)} \Bigr )' = \sup\limits_{t \in (0,3)} \frac{\phi''(t)\phi(t) - \phi'(t)^2}{\phi(t)^2}  < 0,$ \\[0.14in]
\item $\sum_{n \in \mathbb{Z}} \phi(t+n) \equiv 1$
\end{enumerate}
\end{lemma}

\begin{proof}[\bf Proof of Lemma~\ref{L:special-function}]
Consider first the smooth function $ f: \mathbb{R} \rightarrow
\mathbb{R} $, defined by
\begin{equation*}
f(x)=  \left\{
\begin{array}{ll}
 e^{-\frac{2}{x}}, \ \ {\rm  for} \  x > 0 , \\[0.07in]
 0, \ \ \ \ \ \, {\rm for} \ x \leqslant 0
\end{array}  \right.
\end{equation*}
Note that for $x > 0$, one has $$ f''(x) = \frac{4}{x^4}e^{-\frac{2}{x}}(1-x),$$ 
and hence $ f''(x) > 0 $ for $ x \in (0,1) $, and  $
f''(0) = f''(1) = 0 $. Note moreover that $$ f''(x)f(x) - f'(x)^2 =
\Big( \frac{f'(x)}{f(x)} \Big)^{'}f(x)^2 = -\frac{4}{x^3}
e^{-\frac{4}{x}} < 0, \ \ {\rm for} \ x \in (0, +\infty) $$
We approximate, in the $ C^0$-norm, the function $ f''|_{[0,1]} $ arbitrarily close
by a smooth positive function $ h:[0,1] \rightarrow [0,\infty)$, such that
$ h(x) = f''(x)$  for $ x \in [0,\frac{1}{2}] $, and such that $
h(x)=0$
near $ x=1 $. 
Next, consider the smooth function $ F : [0,1] \rightarrow \mathbb{R} $,
that is uniquely determined by the requirements $ F''(x) = h(x) $,
and $ F(0) = F'(0) = 0 $. Note that the function $F$ is arbitrary
close, in the $C^2$-topology, to $ f|_{[0,1]} $, and $ F(x) = f(x) $
for $ x \in [0, \frac{1}{2}] $. Moreover, the requirement that $h$
is $ C^0 $-sufficiently close to $ f''|_{[0,1]} $ ensures that $
F''(x)F(x) - F'(x)^2 < 0 $, for every $ x \in (0,1) $. We further
observe that by definition, $ F''(x) + F''(1-x) > 0 $ for all $ x
\in (0,1) $, and that $ F(x) $ is a linear function near $x=1$.
Finally, we define $ \phi:\mathbb{R} \rightarrow
\mathbb{R} $ as follows:
\begin{equation*}
\phi(x) = \left\{
\begin{array}{llll}
{ \frac{F(x)}{2F(1)} }& \text{for } x \in [0,1],\\[0.1in]
\frac{2F(1) - F(x-1) - F(2-x)}{2F(1)} & \text{for } x \in  (1,2], \\[0.1in]
\frac{F(3-x)}{2F(1)} & \text{for } x \in  (2,3], \\[0.1in]
0 & \text{for } x \notin [0,3]\\[0.05in]
\end{array} \right.
\end{equation*}
It follows immediately from the definition that $\phi$ is a
non-negative smooth function, with  $ supp (\phi) = [0,3] $. Note
moreover that $ \phi(x) = \phi(3-x) $, and that for $x \in (1,2)$:
\begin{equation} \label{eq-about-derivative-of-phi}
(\phi_{\arrowvert_{(1,2)}})''(x) = { {\frac {-F''(x-1)-F''(2-x)}
{2F(1)}}} < 0 \end{equation}
Combining this with the fact that $\phi'(3/2)=0$, we obtain that $
\phi'(x)
> 0 $ for $ x \in (1,3/2) $, and $ \phi'(x) < 0 $ for $ x \in
(3/2,3) $. Furthermore, from the definition of the function
$F$, it follows that
$ \phi'(x) > 0 $ for $ x \in (0,1] $ and $ \phi'(x) < 0 $ for $ x
\in [2,3) $. Thus, we conclude that $\phi$ satisfies the first three
requirements of the lemma.
We next turn to show that $\phi$ satisfies the forth one.
Note that $\phi''(x)\phi(x) - \phi'(x)^2 < 0 $ for $x
\in (0,3)$. This follows from the analogous property of $ F $ for $
x \in (0,1) \cup (2,3) $; from~$(\ref{eq-about-derivative-of-phi})$
for $x \in (1,2)$; and from the fact that $ \phi''(x_0)\phi(x_0) -
\phi'(x_0)^2 = -\phi'(x_0)^2 < 0 $ for $ x_0 = 1,2 $.
Moreover, from the definition of the function $\phi $ it follows
that  $ \phi(x) \simeq e^{-\frac{2}{x}} $ for $ x $ close to $ 0 $,
and $ \phi(x) \simeq e^{-\frac{2}{3-x}} $ for $ x $ close to $ 3 $,
where $\simeq$ means arbitrary close in the $C^2$-topology.
Therefore, we obtain:
$$ \lim_{x \rightarrow 0^+} \frac{\phi''(x)\phi(x) -
\phi'(x)^2}{\phi(x)^2} = \lim_{x \rightarrow 3^-}
\frac{\phi''(x)\phi(x) - \phi'(x)^2}{\phi(x)^2} = - \infty .$$ From
the above we conclude that: $$ \sup_{x \in (0,3)}
\frac{\phi''(x)\phi(x) - \phi'(x)^2}{\phi(x)^2} < 0,$$ as required.
Finally, there is no difficulty in checking that $ \sum_{n \in
\mathbb{Z }} \phi(x+n) = 1 $. The details of this last step are left
to the reader.
\end{proof}

\begin{lemma} \label{L:Localization1}
 Let $ R= [\alpha_1,\beta_1] \times [\alpha_2,\beta_2] \subset \mathbb{R}^2 $ be a rectangle, and  consider two smooth non-negative functions
$u : [\alpha_1,\beta_1] \rightarrow {\mathbb R}$, and $v :
[\alpha_2,\beta_2] \rightarrow {\mathbb R}$, positive on $ (\alpha_1,\beta_1) $ and $ (\alpha_2,\beta_2) $ respectively, such that $u(x) = e^{{\frac {-1} {x-\alpha_1}}}$ near $\alpha_1$; $u(x) = e^{{\frac
{-1} {\beta_1-x}}}$ near $\beta_1$, $v(y) =e^{{\frac {-1}
{y-\alpha_2}}}$ near $\alpha_2$; and $v(y) = e^{{\frac {-1}
{\beta_2-y}}}$ near $\beta_2$. Moreover, let $ \phi(x) $ be the
function described in Lemma~\ref{L:special-function} above, and let
 $ F: \mathbb{R}^2 \rightarrow
\mathbb{R} $ be any smooth function that satisfies:
\begin{packed_enum}
\item $ supp(F) = R $ 
\item $ F(x,y) > 0 $ for $ (x,y) \in int(R) $
\item $F(x,y) = u(x)v(y)$ near the boundary of $R$ 
\end{packed_enum}
Then there exists an $ \epsilon_0 > 0 $, such that for any $ 0 <
\epsilon < \epsilon_0 $, and $ a \in \mathbb{R} $, the following holds:
denote by $ G(x,y) = F(x,y) \phi(\frac{x-a}{\epsilon}) $, and assume that $G
\neq 0$ (this holds when $ (\alpha_1,\beta_1) \cap (a,a+3\epsilon) \neq \emptyset $). Moreover,
set $ U = supp(G) = [a_1,a_2] \times [\alpha_2,\beta_2] $. Then, there
exists a smooth function $ c : [\alpha_2,\beta_2] \rightarrow
(a_1,a_2) $, which is constant near $ \alpha_2,\beta_2 $, such that
for any $ y \in (\alpha_2,\beta_2) $ one has:
\begin{equation*}
 \left\{
\begin{array}{rl}
 \frac{\partial}{\partial x} G(x,y) > 0,  & \text{for } a_1 < x <
c(y),\\[0.1in]
\frac{\partial}{\partial x} G(x,y) < 0, & \text{for }  c(y) < x <
a_2
\end{array} \right.
\end{equation*}

\end{lemma}

\begin{proof}[{\bf Proof of Lemma~\ref{L:Localization1}}]
From the above assumptions it follows that
there exists $ \alpha_1 < \gamma_1 < \delta_1 < \beta_1 $, such that
$ u(x) = e^{{\frac {-1} {x-\alpha_1}}}$ for $ \alpha_1 < x <
\gamma_1 $, $u(x) = e^{{\frac {-1} {\beta_1-x}}}$ for $\delta_1 < x
< \beta_1$, and $ F(x,y) = u(x)v(y) $ when $ x \in
(\alpha_1,\gamma_1] \cup [\delta_1,\beta_1) $. Pick some $
\gamma_1',\delta_1' $, such that $ \alpha_1 < \gamma_1' < \gamma_1 <
\delta_1 < \delta_1' < \beta_1$, and denote $ \epsilon_1 = \min \{
\frac{\gamma_1 - \gamma_1'}{3},\frac{\delta_1' - \delta_1}{3} \} $.
 Next, take any $ 0 < \epsilon < \epsilon_1 $, and any $ a \in \mathbb{R} $, and
consider the function $ G(x,y) = F(x,y) \phi(\frac{x-a}{\epsilon})$.

\noindent {\bf Case I:} Assume $ a \in [\gamma_1',\delta_1] $. Then,
one has $ \gamma_1' \leqslant a < a+3\epsilon \leqslant \delta_1' $,
and therefore $ supp(G) = [a,a+3\epsilon] \times [\alpha_2,\beta_2]
$. Fix some $ y_0 \in (\alpha_2,\beta_2) $. Our goal is to show that
for sufficiently small $ \epsilon $ (which is independent of $ y_0
$), there exists a value $ c(y_0) \in (a,a+3\epsilon) $, such that $
\frac{\partial}{\partial x} G(x,y_0) > 0 $, for $ a < x < c(y_0) $,
and $ \frac{\partial}{\partial x} G(x,y_0) < 0 $, for $ c(y_0) < x <
a + 3\epsilon $. For this end, we compute:  $$ \frac{ \frac{\partial}{\partial x}
G(x,y_0) }{G(x,y_0)} = \frac{ \frac{\partial }{\partial x} F(x,y_0)
}{F(x,y_0)} + \frac{1}{\epsilon} \frac{ \phi'( \frac{x-a}{\epsilon}
) }{ \phi( \frac{x-a}{\epsilon} ) } .$$ Note that, the function $ x
\mapsto G(x,y_0) $ is a 
positive function, supported in
$ [a,a + 3\epsilon] $. Thus, $
\frac{ \frac{\partial}{\partial x} G(x,y_0) }{G(x,y_0)} = 0 $ at
least at one point $ x \in (a, a+3\epsilon) $ (e.g., at the
maximum point of $ x \mapsto G(x,y_0) $). Let us show next that: 
\begin{equation} \frac{\partial}{\partial x} \frac{ \frac{\partial}{\partial x} G(x,y_0) }{G(x,y_0)} < 0, \ \ {\rm for \ all \ } x \in (a,a+3\epsilon) \end{equation} 
We start by claiming that $ \frac{\partial}{\partial x} \frac{\frac{\partial }{\partial x} F(x,y) }{F(x,y)} $ is bounded on $ [\gamma_1',\delta_1'] \times (\alpha_2,\beta_2) $. Indeed, from the assumptions of the lemma it follows that  $ F(x,y) = u(x)v(y) $ near the boundary of $ R $, and therefore there exist $ \alpha_2 < \gamma_2 < \delta_2 < \beta_2 $, such that $ F(x,y) = u(x)v(y) $ for $ y \in [\alpha_2,\gamma_2] \cup  [\delta_2,\beta_2]$. 
Thus, for a point $(x,y)$ near the boundary of $R$, one has
\begin{equation} \frac{\partial}{\partial x} \frac{\frac{\partial }{\partial x} F(x,y) }{F(x,y)} =
\frac{\partial}{\partial x} \frac{u'(x)}{u(x)} = \frac{u''(x)u(x) -
u'(x)^2}{u(x)^2} \end{equation} Restricting ourselves to the case where $ x \in
[\gamma_1',\delta_1'] $ and $ y \in [\alpha_2,\gamma_2] \cup  [\delta_2,\beta_2]$, and by noticing that $
u|_{(\alpha_1,\beta_1)} $ is strictly positive smooth function, we
obtain that the function $ \frac{\partial}{\partial x}
\frac{\frac{\partial }{\partial x} F(x,y) }{F(x,y)} $ is bounded on
$ [\gamma_1',\delta_1'] \times ((\alpha_2,\gamma_2] \cup
[\delta_2,\beta_2))$. On the other hand, because of compactness,
the function $ \frac{\partial}{\partial x} \frac{\frac{\partial
}{\partial x} F(x,y) }{F(x,y)} $ is bounded on $
[\gamma_1',\delta_1'] \times [\gamma_2,\delta_2] $. 
Hence, we
conclude that $ \frac{\partial}{\partial x} \frac{\frac{\partial
}{\partial x} F(x,y) }{F(x,y)} $ is bounded on 
$ [\gamma_1',\delta_1'] \times
(\alpha_2,\beta_2) $. Next, note that \begin{equation} \label{eq-about-deriv-of-phiprime-over-phi} \frac{\partial}{\partial
x} \frac{1}{\epsilon} \frac{ \phi'( \frac{x-a}{\epsilon} ) }{ \phi(
\frac{x-a}{\epsilon} ) } = \frac{1}{\epsilon^2} \frac{ \phi''(
\frac{x-a}{\epsilon}) \phi( \frac{x-a}{\epsilon}) - \phi'(
\frac{x-a}{\epsilon})^2 }{ \phi( \frac{x-a}{\epsilon} )^2 } \end{equation} From
Lemma~\ref{L:special-function} it follows that
\begin{equation} \label{eq-prop-number-4-from-the-tech-lemma}  \sup_{t\in(0,3)} \frac{ \phi''(t) \phi(t) - \phi'(t)^2 }{ \phi(t)^2 } < 0, \end{equation}
and hence~$(\ref{eq-about-deriv-of-phiprime-over-phi})$ can be chosen to be arbitrarily negative.
As a conclusion, we obtain that for sufficiently small $ \epsilon $,
say $ 0 < \epsilon < \epsilon_2 $, one has
\begin{equation} \label{eq-about-the-der-of-g-over-gprime}
\frac{\partial}{\partial x} \frac{ \frac{\partial}{\partial x}
G(x,y) }{G(x,y)} < 0, \  {\rm for \ every \ } (x,y) \in supp(G) = [a,a+3\epsilon]
\times (\alpha_2,\beta_2) \end{equation} Moreover, for any $ y \in
(\alpha_2,\beta_2) $, there exists therefore a unique $ x:= c(y)
\in (a,a+3\epsilon) $, such that $ \frac{ \frac{\partial}{\partial
x} G(x,y) }{G(x,y)} = 0 $. It follows from~$(\ref{eq-about-the-der-of-g-over-gprime}$)
and the implicit function theorem, that
the function $ y \mapsto c(y) $ is smooth for $ y \in
(\alpha_2,\beta_2) $. Moreover, since  $ \frac{
\frac{\partial}{\partial x} G(x,y) }{G(x,y)} $ is independent of $ y
$, when $ y $ is close to $ \alpha_2 $ or to $ \beta_2 $, it follows
that $ y \mapsto c(y) $ is constant near the endpoints
$ \alpha_2,\beta_2 $.
This completes the proof of the Lemma in Case I. \\

\noindent {\bf Case II}: Assume that $ a < \gamma_1' $ or $ a > \delta_1 $.
Here we have $ [a,a + 3 \epsilon] \subset (-\infty,\gamma_1)
\cup (\delta_1,+\infty) $. Therefore, the function
$$ \frac{\frac{\partial }{\partial x} F(x,y) }{F(x,y)} =  \frac{u'(x)}{u(x)} $$ is independent of $ y $, as well as
$$ \frac{\frac{\partial}{\partial x} G(x,y) }{G(x,y)} = \frac{ u'(x) }{ u(x) } +
\frac{1}{\epsilon} \frac{ \phi'( \frac{x-a}{\epsilon} ) }{ \phi(
\frac{x-a}{\epsilon} ) } ,$$ for $ (x,y) \in supp(G) $. Also, for $
(x,y) \in supp(G) $ one has $$ \frac{\partial}{\partial x}
\frac{\frac{\partial }{\partial x} F(x,y) }{F(x,y)} =
\frac{\partial}{\partial x} \frac{u'(x)}{u(x)} = \frac{u''(x)u(x) -
u'(x)^2}{u(x)^2}$$ Thus, since $ u(x) = e^{{\frac {-1}
{x-\alpha_1}}}$ for $x \in (\alpha_1,\gamma_1) $, and $u(x) = e^{{\frac
{-1} {\beta_1-x}}}$ for $x \in (\delta_1,\beta_1) $, we obtain
$$ \frac{\partial}{\partial x} \frac{\frac{\partial
}{\partial x} F(x,y) }{F(x,y)} < 0,  \ \ {\rm for } \ (x,y) \in supp(G) $$ 
As in Case I, by combining~$(\ref{eq-about-deriv-of-phiprime-over-phi})$ and~$(\ref{eq-prop-number-4-from-the-tech-lemma})$, one has $$ \frac{\partial}{\partial x} \frac{1}{\epsilon}
\frac{ \phi'( \frac{x-a}{\epsilon} ) }{ \phi( \frac{x-a}{\epsilon} )
} < 0,  \ \ {\rm for } \ (x,y) \in supp(G)  $$ 
Therefore, we conclude that $$
\frac{\partial}{\partial x} \frac{ \frac{\partial}{\partial x}
G(x,y) }{G(x,y)} < 0,  \ \ {\rm for } \ (x,y) \in supp(G)  $$ 
As in the previous case,
since $ x \rightarrow G(x,y_0) $ is positive in the interior of its  support
$ supp(G) = [a_1,a_2] \times [\alpha_2,\beta_2] $, for each fixed $
y_0 \in (\alpha_2,\beta_2) $ there exists $ x \in (a_1,a_2) $ such
that $ \frac{\partial}{\partial x} G(x,y_0) = 0 $. Therefore for
each fixed $ y_0 \in (\alpha_2,\beta_2)$, there is a unique $ x =c(y_0) \in
(a_1,a_2)$, such that $ \frac{ \frac{\partial}{\partial x} G(x,y)
}{G(x,y)} = 0 $. Moreover, since the function $
\frac{\frac{\partial}{\partial x} G(x,y) }{G(x,y)} $ is independent
of $ y $ for $ (x,y) \in supp(G) $, we conclude that the function $ y
\mapsto c(y) $ is constant on $ (\alpha_2,\beta_2) $. This completes
the proof of 
lemma~\ref{L:Localization1}.
\end{proof}

\begin{lemma} \label{L:Localization2}
In the same setting as in Lemma~\ref{L:Localization1}, for any
open neighborhood $ V$ of $ U = supp(G) = [a_1,a_2] \times
[\alpha_2,\beta_2] $, there exists a compactly supported
diffeomorphism $ \Phi : V \rightarrow V $, such that $ H = \Phi^*G $ satisfies
$ | \frac{\partial}{\partial x} H | \leqslant
\frac{3 \|G\|_{\infty}}{a_2 - a_1} $, and $supp(H) = supp(G)$.
\end{lemma}

\begin{proof}[{\bf Proof of Lemma~\ref{L:Localization2}}]

We divide the proof of the lemma into two steps: \\

\noindent {\bf Step I: }
Let $V$ 
be an open neighborhood of $ U = supp(G) = [a_1,a_2] \times
[\alpha_2,\beta_2] $. Take $ \widetilde{\alpha}_2 < \alpha_2 < \beta_2 < \widetilde{\beta_2} $, such that $ [a_1,a_2] \times [\widetilde{\alpha}_2,\widetilde{\beta}_2] \subset V $.
Moreover, take $ \widetilde{a}_1,\widetilde{a}_2 $ such that $$ a_1 < \widetilde{a}_1 < \min_{[\alpha_2,\beta_2]} c(y) \leq \max_{[\alpha_2,\beta_2]} c(y) < \widetilde{a}_2 < a_2,$$
and, $$ \widetilde{a}_1 < \frac{a_1+a_2}{2} < \widetilde{a}_2 .$$
One can easily find a smooth family of diffeomorphisms $ f^t : (a_1,a_2) \rightarrow (a_1,a_2) $, $ t \in ( \widetilde{a}_1 , \widetilde{a}_2 ) $, such that:
\begin{equation*}
 \left\{
\begin{array}{lll}
 supp(f^t) \subset [ \widetilde{a}_1 , \widetilde{a}_2 ], \\[0.1in]
f^t(\frac{a_1+a_2}{2}) = t, \\[0.1in]
f^{ \frac{a_1+a_2}{2} } = \Id_{(a_1,a_2)}
\end{array} \right.
\end{equation*}
We extend the function $ c(y) $ to a smooth function on the interval $
(\widetilde{\alpha}_2,\widetilde{\beta}_2) $, such that $ c(y) = \frac{a_1+a_2}{2},$ for $y$ close enough to the points $\widetilde{\alpha}_2,\widetilde{\beta}_2 $. Next, define a diffeomorphism $$ \Psi_1 : (a_1,a_2) \times (\widetilde{\alpha}_2,\widetilde{\beta}_2) \rightarrow (a_1,a_2) \times (\widetilde{\alpha}_2,\widetilde{\beta}_2) $$ by the requirement: $$ \Psi_1(x,y) = (f^{c(y)}(x),y) .$$ It is not hard to check that the diffeomorphism $ \Psi_1 $ is the identity near the boundary of the rectangle $ (a_1,a_2) \times (\widetilde{\alpha}_2,\widetilde{\beta}_2) $, and therefore one can extend it by the identity, allowing ourselves a slight abuse of notation, to a diffeomorphism $ \Psi_1 : V \rightarrow V $. Denote $ G_1 = \Psi_1^* G $.
It follows from the definition of $ \Psi_1 $ that for $ y \in (\alpha_2,\beta_2) $, one has:
\begin{equation} \label{estimate-the-x-derivative-of-G1}
 \left\{
\begin{array}{rl}
 \frac{\partial}{\partial x} G_1(x,y) > 0,  & \text{for } a_1 < x <
\frac{a_1+a_2}{2},\\[0.1in]
\frac{\partial}{\partial x} G_1(x,y) < 0, & \text{for }  \frac{a_1+a_2}{2} < x <
a_2,
\end{array} \right.
\end{equation}
and moreover that $ supp (G_1) = [a_1,a_2] \times [\alpha_2,\beta_2] $, and $ G_1(x,y) = u_1(x)v_1(y) $ for $ x \in [a_1,a_2] $ and $ y $ being near $ \alpha_2 $ or $ \beta_2 $, where $ u_1(x) = (f^{c(\alpha_2)})^* ( u(x) \phi(\frac{x-a}{\epsilon})) $, $ v_1(x) = v(x) $.

\noindent {\bf \\ Step II: }

Let $0< \epsilon < {
\frac {a_2-a_1} {10}}$, and consider three families of smooth positive functions $ \chi^{\epsilon}_j : [a_1,a_2] \rightarrow [0,1]$, where $j = 1,2,3$, such that the following holds:  
\begin{eqnarray*}
& \chi^{\epsilon}_1(x) &  =  \left\{
\begin{array}{rl}
 1,  & \text{for } x \in [a_1,a_1+\epsilon] \cup [\frac{a_1+a_2}{2} - \epsilon, \frac{a_1+a_2}{2} + \epsilon ] \cup [a_2 - \epsilon,a_2] ,\\[0.1in]
0, & \text{for }  x \in [a_1+2\epsilon,\frac{a_1+a_2}{2} - 2\epsilon] \cup  [ \frac{a_1+a_2}{2} + 2\epsilon , a_2 - 2\epsilon],
\end{array} \right. \\
%
& \chi^{\epsilon}_2(x) & = \left\{
\begin{array}{lll}
 0,  & \text{for }  x \in [a_1,a_1+\epsilon] \cup [\frac{a_1+a_2}{2} - \epsilon, a_2] ,\\[0.1in]
1 , & \text{for }  x \in [a_1+2\epsilon,\frac{a_1+a_2}{2} - 2\epsilon],
\end{array} \right. \\
& \chi^{\epsilon}_3(x) & = \left\{
\begin{array}{rl}
 0,  & \text{for }  x \in [a_1,\frac{a_1+a_2}{2} + \epsilon] \cup [a_2 - \epsilon, a_2]  ,\\[0.1in]
1 , & \text{for }  x \in [\frac{a_1+a_2}{2} + 2\epsilon, a_2 - 2\epsilon],
\end{array} \right.
\end{eqnarray*}
and moreover, 
\begin{equation*}
 \left\{
\begin{array}{rl}
 \chi^{\epsilon}_2(x) > 0,  & \text{for } x \in (a_1 + \epsilon, \frac{a_1+a_2}{2} - \epsilon)  ,\\[0.1in]
\chi^{\epsilon}_3(x) > 0, & \text{for }   x \in (\frac{a_1+a_2}{2} + \epsilon, a_2 - \epsilon).
\end{array} \right.
\end{equation*}

Next, denote by $ C_0^\infty ([a_1,a_2]) $ the set of smooth functions $ [a_1,a_2] \rightarrow \mathbb{R} $, such that
the derivatives of any order (including zero) vanish at the boundary points $ a_1$ and $a_2 $.
Fix $ g \in C_0^\infty ([a_1,a_2]) $, and define $ h_{\epsilon}(x)$ by:
$$ h_{\epsilon}(x) = g'(x)\chi_1^{\epsilon}(x) + A\chi_2^{\epsilon}(x) - B\chi_3^{\epsilon}(x),$$
where $A$ and $B$ are two constants given by:
$$ A = \frac{
g(\frac{a_1+a_2}{2}) - \int_{a_1}^{\frac{a_1+a_2}{2}} g'(x) \chi_1^{\epsilon}(x) \, dx  }{ \int_{a_1}^{\frac{a_1+a_2}{2}} \chi_2^{\epsilon}(x) \, dx } ,$$ and
$$ B = \frac{ g(\frac{a_1+a_2}{2}) + \int_{\frac{a_1+a_2}{2}}^{a_2} g'(x) \chi_1^{\epsilon}(x) \, dx  }{ \int_{\frac{a_1+a_2}{2}}^{a_2} \chi_3^{\epsilon}(x) \, dx } .$$
Note that one has: $$ \int_{a_1}^{\frac{a_1+a_2}{2}} h_\epsilon (x) dx = g(\frac{a_1+a_2}{2}) ,$$ and  $$ \int_{\frac{a_1+a_2}{2}}^{a_2} h_\epsilon (x) dx = - g(\frac{a_1+a_2}{2}) .$$ Let $ g_\epsilon : [a_1,a_2] \rightarrow \mathbb{R} $ be the unique function such that $ g_\epsilon '(x) = h_\epsilon(x) $, and $ g_\epsilon(a_1) = 0 $.
It follows from the definition that
$$g_\epsilon(x) = g(x), \ \ {\rm for} \ \ x \in [a_1,a_1+\epsilon] \cup [\frac{a_1+a_2}{2} - \epsilon, \frac{a_1+a_2}{2} + \epsilon ] \cup [a_2 - \epsilon,a_2], $$ and in particular, $ g_\epsilon \in C_0^\infty ([a_1,a_2]) $.
Note moreover that if $ g(x) $ satisfies $ g'(x) > 0 $ for $ x \in (a_1,\frac{a_1+a_2}{2}) $ and $ g'(x) < 0 $ for $ x \in (\frac{a_1+a_2}{2},a_2) $, then so is $ g_\epsilon(x) $ i.e., $ g_\epsilon '(x) > 0 $ for $ x \in (a_1,\frac{a_1+a_2}{2}) $ and $ g_\epsilon '(x) < 0 $ for $ x \in (\frac{a_1+a_2}{2},a_2) $.

Next, 
we define a family of operators $ L_\epsilon : C_0^\infty ([a_1,a_2]) \rightarrow C_0^\infty ([a_1,a_2]) $, by 
the requirement that $ L_\epsilon g = g_\epsilon $. It is not hard to check that $ L_\epsilon $ is linear, and continuous in the $ C^\infty $-topology. Moreover, let
$$ {\cal I}_{\epsilon} := [a_1,a_1+2\epsilon] \cup [\frac{a_1+a_2}{2} - 2\epsilon, \frac{a_1+a_2}{2} + 2\epsilon ] \cup [a_2 - 2\epsilon,a_2]$$
Then, from the definition of $g_{\epsilon}$, and the fact that $\chi_2^{\epsilon}$ and $\chi_3^{\epsilon}$ has disjoint support, one has the following estimate:
  $$ \max_{[a_1,a_2]} |g_\epsilon '(x)| \leqslant \max_{x \in {\cal I}_{\epsilon}  } |g'(x)| + \max \{|A|, |B|\} .$$
Furthermore, from the definition of $A$ and $B$ one has:
$$ |A| , |B| \leqslant \frac{ |g(\frac{a_1+a_2}{2})| + 4\epsilon \max_{ x \in {\cal I}_{\epsilon}   } |g'(x)| }{ \frac{a_2 - a_1}{2} - 4 \epsilon} .$$
Therefore, we conclude that
\begin{equation} \label{eq-estimate-on-max-of-g-eps} \max_{[a_1,a_2]} |g_\epsilon '(x)| \leqslant \frac{ |g(\frac{a_1+a_2}{2})| }{ \frac{a_2 - a_1}{2} - 4 \epsilon }
\, + \, \left( 1 + \frac{4\epsilon}{ \frac{a_2 - a_1}{2} - 4 \epsilon } \right) \max_{ x \in {\cal I}_{\epsilon}   } |g'(x)| . \end{equation}
Next, define
$ H_\epsilon \colon [a_1,a_2] \times [ \alpha_2, \beta_2 ] \rightarrow \mathbb{R} $ by $ H_\epsilon (\cdot,y) = L_\epsilon G_1(\cdot,y) $ for every $ y \in [ \alpha_2 , \beta_2 ] $. Note that
$ H_{\epsilon}|_{ [a_1,a_2] \times [ \alpha_2, \beta_2 ]} $ is a smooth function.
Moreover, if $ \epsilon > 0 $ is small enough, then from~$(\ref{eq-estimate-on-max-of-g-eps})$ we conclude that $$ | \frac{\partial}{\partial x} H_\epsilon (x,y) | \leqslant
 \frac{ {3\| G_1 \|_{\infty}} }{a_2 - a_1} =  \frac{ {3 \| G \|_{\infty} }}{a_2 - a_1}, \ \ {\rm for \ every \ } (x,y) \in [a_1,a_2] \times [ \alpha_2, \beta_2 ]$$ 
We fix such an $ \epsilon $, and set $ H := H_\epsilon $. From the definition of $H$ and~$(\ref{estimate-the-x-derivative-of-G1})$ one has:
\begin{equation} \label{eq-H-property}
 \left\{
\begin{array}{rl}
 \frac{\partial}{\partial x} H (x,y) > 0,  & \text{for } a_1 < x <
\frac{a_1+a_2}{2},\\[0.1in]
\frac{\partial}{\partial x} H (x,y) < 0, & \text{for }  \frac{a_1+a_2}{2} < x <
a_2,
\end{array} \right.
\end{equation}
for any $ y \in (\alpha_2,\beta_2) $. Furthermore,
\begin{equation} \label{eq-H-G_1}
 H(x,y) = G_1(x,y)
\end{equation}
for $ x \in [a_1,a_1+\epsilon] \cup [\frac{a_1+a_2}{2} - \epsilon, \frac{a_1+a_2}{2} + \epsilon ] \cup [a_2 - \epsilon,a_2] $, and $ y \in (\alpha_2,\beta_2) $. Note moreover that since the operator $ L_\epsilon $ is linear, one has that $ H(x,y) = (L_\epsilon u_1)(x) v_1(y) $ for any $ x \in [a_1,a_2] $ and $y$ being near the boundary points $ \alpha_2 $ or $ \beta_2 $.

 It follows from~$(\ref{eq-H-property})$ and~$(\ref{eq-H-G_1})$ above, that there is a unique diffeomorphism $ \Psi_2 \colon (a_1,a_2) \times (\alpha_2,\beta_2) \rightarrow (a_1,a_2) \times (\alpha_2,\beta_2) $, of the form $ \Psi_2(x,y) = (w(x,y),y) $, such that $$ H|_{(a_1,a_2) \times (\alpha_2,\beta_2)} = \Psi_2^* G_1|_{(a_1,a_2) \times (\alpha_2,\beta_2)} ,$$ and 
 $ supp(\Psi) \subset \left( [a_1+\epsilon,\frac{a_1+a_2}{2} - \epsilon] \cup [ \frac{a_1+a_2}{2} + \epsilon , a_2 - \epsilon ] \right) \times [\alpha_2,\beta_2] $.
Moreover, we have $ G_1(x,y) = u_1(x) v_1(y) $, $ H(x,y) = (L_\epsilon u_1)(x) v_1(y) $ for $ x \in [a_1,a_2] $ and $ y $ being near $ \alpha_2 $ or $ \beta_2 $. From this we conclude that $ w(x,y) $ is independent of $ y $, for $ y $ being close to $ \alpha_2,\beta_2 $. From Step I, we have $ \widetilde{\alpha}_2 < \alpha_2 < \beta_2 < \widetilde{\beta_2} $, such that $ [a_1,a_2] \times [\widetilde{\alpha}_2,\widetilde{\beta}_2] \subset V $. One can easily extend the diffeomorphism $ \Psi_2 $ to  $$ \Psi_2 : (a_1,a_2) \times (\widetilde{\alpha}_2,\widetilde{\beta}_2) \rightarrow (a_1,a_2) \times (\widetilde{\alpha}_2,\widetilde{\beta}_2) ,$$ such that $ \Psi_2 $ is the identity diffeomorphism near the boundary of $ (a_1,a_2) \times ( \widetilde{\alpha}_2,\widetilde{\beta}_2) $. Then we can extend $ \Psi_2 $ by the identity to be a diffeomorphism $ \Psi_2 : V \rightarrow V $. We have $ H = \Psi_2^* G_1 $.


Finally, denote $ \Phi = \Psi_1 \Psi_2 \colon V \rightarrow V$. The diffeomorphism $\Phi$ is compactly supported inside $ V $, and
$ H = \Phi^* G$ satisfies $$ | \frac{\partial}{\partial x} H | \leqslant
\frac{3 \|G\|_{\infty}}{a_2 - a_1}, \ \ {\rm and} \ \  supp(H) = supp(G)$$
This completes the proof of the lemma.
\end{proof}

We are finally in a position to prove Lemma~\ref{lemma-about-decomposition}.

\begin{proof}[{\bf Proof of Lemma~\ref{lemma-about-decomposition}}]
Let $f : {\mathbb R}^2 \rightarrow {\mathbb R}$ be a smooth function with $\|f\|_{\infty} \leqslant 1$,
and $supp(f) \subset int(R)$. 
We fix some parameters $\alpha_i,\alpha_i',\beta_i,\beta_i'$, where $i=1,2$, such that $0 < \alpha_i < \alpha'_i< \beta_i' < \beta_i < L_i$, for $i=1,2$; $ \beta_1 - \alpha_1 >
\frac{3}{4} L_1 $; and
 $$supp(f) \subset int( [\alpha'_1,\beta'_1]
\times [\alpha'_2,\beta'_2]) \subset int( [\alpha_1,\beta_1] \times
[\alpha_2,\beta_2]) \subset int(R)$$
Moreover, we choose a smooth function $u: [0,L_1] \rightarrow {\mathbb R}$,
such that $u(x) = e^{{\frac {-1} {x-\alpha_1}}}$ near
$\alpha_1$, $u(x) = e^{{\frac {-1} {\beta_1-x}}}$ near $\beta_1$, $u(x) = 1$ on $[\alpha_1',\beta_1']$, and
$ \| u \|_{\infty} = 1 $.
Similarly, we take $v: [0,L_2] \rightarrow {\mathbb R}$, with $v(y) = e^{{\frac {-1} {y-\alpha_2}}}$ near
$\alpha_2$, $v(y) = e^{{\frac {-1} {\beta_2-y}}}$ near $\beta_2$, $v(y) = 2$ on $[\alpha_2',\beta_2']$, and
$ \| v \|_{\infty} = 2 $. Next, we consider the decomposition $f = F_1-F_2$, where
$$F_1(x,y) = f(x,y)+u(x)v(y), \ {\rm  and } \ \ F_2(x,y)=u(x)v(y)$$
We have $ \| F_{\varsigma}(x,y) \|_{\infty} \leqslant 3 $ for $\varsigma \in \{1,2\}$.
From 
Lemma~\ref{L:Localization2} it follows that there is $\epsilon_0 > 0$ such that for any $0 < \epsilon < \epsilon_0$, and
 any $a \in {\mathbb R}$, the following holds: let $ G_{\varsigma}(x,y) =
F_{\varsigma}(x,y) \phi(\frac{x-a}{\epsilon}) $, where $\varsigma \in \{1,2\}$ (we may and shall assume in what follows that
$G_{\varsigma} \neq 0 $). Take $V^{\varsigma}$ to be any open neighborhood of $U^{\varsigma} := supp(G_{\varsigma}) =
[a_1^{\varsigma},a_2^{\varsigma}] \times [\alpha_2,\beta_2] $. Then, there is a compactly
supported diffeomorphism $ \Phi^{\varsigma} : V^{\varsigma} \rightarrow V^{\varsigma} $, such that $ H_{\varsigma}
= (\Phi^{\varsigma})^*G_{\varsigma} $ satisfies \begin{equation} \bigl  | {\textstyle \frac{\partial}{\partial x}} H_{\varsigma} \bigr | \leqslant {\textstyle \frac{9}{a_2^{\varsigma} - a_1^{\varsigma}}}, \ {\rm and } \ supp(H_{\varsigma}) = supp(G_{\varsigma}) \end{equation}

Fix  $0 <  \epsilon < \epsilon_0$ as above. For $ n \in \mathbb{Z} $ and $\varsigma \in \{1,2\}$ denote $
G_{\varsigma,n} = F_{\varsigma}(x,y) \phi(\frac{x-n\epsilon}{\epsilon}) $. Note that $
F_{\varsigma} = \sum_{n \in \mathbb{Z}} G_{\varsigma,n} $, and that only finitely many summands
 are not identically zero.
For $
i=1,2,3,4 $,  let $ K_{\varsigma,i} = \sum_{j \in \mathbb{Z}} G_{\varsigma,i+4j} $.
Note moreover that the supports of all the non-zero summands of $ K_{\varsigma,i}$ are pairwise disjoint, and 
$ F_{\varsigma} = \sum_{i=1}^4 K_{\varsigma,i}$, and thus $ f = \sum_{\varsigma=1}^2 \sum_{i=1}^4 K_{\varsigma,i}$.
Next, we fix  $ 1
\leqslant i_0 \leqslant 4 $. Consider 
$ K_{\varsigma,i_0}
= \sum_{j \in \mathbb{Z}} G_{\varsigma,i_0 + 4j} $, and choose pairwise
disjoint open neighborhoods $ V^{\varsigma}_{i_0,j} \supset supp( G_{\varsigma,i_0 + 4j} ) $ of those summands
which are not identically zero. Now, apply  
Lemma~\ref{L:Localization2} to each element in the decomposition $
K_{\varsigma,i_0} = \sum_{j \in \mathbb{Z}} G_{\varsigma,i_0 + 4j} $. We obtain that for
any non-zero summand $ G_{\varsigma,i_0 + 4j} $, there is a compactly
supported diffeomorphism $ \Phi^{\varsigma}_{i_0,j} : V^{\varsigma}_{i_0,j} \rightarrow V^{\varsigma}_{i_0,j} $, such that
the function $ H^{\varsigma}_{i_0,j} = (\Phi^{\varsigma}_{i_0,j})^* \,G_{\varsigma,i_0 + 4j} $ satisfies
\begin{equation} {\displaystyle \left |
\tfrac{\partial}{\partial x} H^{\varsigma}_{i_0,j} \right | \leqslant
{\displaystyle \frac{9}{{\textstyle   \mu(\pi_x(supp(G_{\varsigma,i_0 + 4j})))} }  }}, \ {\rm  and} \
supp(H^{\varsigma}_{i_0,j}) = supp(G_{\varsigma,i_0 + 4j}) \end{equation}
Here $\pi_x$ denotes the projection to the interval $[0,L_1]$, and $\mu$ is the  Lebesgue measure.
Note that the supports $
\{ supp(\Phi^{\varsigma}_{i_0,j}) \}$ are mutually disjoint. We shall denote by $ \widetilde{\Phi}^{\varsigma}_{i_0} $
the composition of all the $ \Phi^{\varsigma}_{i_0,j}$'s for which
 $G_{\varsigma,i_0 + 4j} \neq 0 $. Moreover, we denote by $ \Pi^{\varsigma}_{i_0,k} $, $ k=1,2,...,M_{i_0} $ all
the non-empty supports among $\{ supp(G_{\varsigma,i_0 + 4j}) \}$. Note that each $\Pi^{\varsigma}_{i_0,k}$ is a
rectangle contained in $ [\alpha_1,\beta_1] \times [\alpha_2,\beta_2] $. Consider a sequence of rectangles
$$ \widetilde {\Pi}^{\varsigma}_{i_0,k}  := [\alpha_1,\beta_1] \times [\alpha_2 + (2k-1) \frac{\beta_2-\alpha_2}{2M_{i_0}},\alpha_2 + 2k \frac{\beta_2-\alpha_2}{2M_{i_0}}] $$
It is not hard to check that there exists a diffeomorphism $ \Psi^{\varsigma}_{i_0}: R \rightarrow R $,
such that $\Psi^{\varsigma}_{i_0}(\widetilde {\Pi}^{\varsigma}_{i_0,k}) = \Pi^{\varsigma}_{i_0,k} $, and moreover that on each  $\widetilde {\Pi}^{\varsigma}_{i_0,k} $ it coincides with a linear contraction on the directions
of the axes, composed with a translation. As a
result, for $ k_{\varsigma,i_0} := (\Psi_{i_0}^{\varsigma})^* (\widetilde{\Phi}_{i_0}^{\varsigma})^* K_{\varsigma,i_0}
$, one has $ | \frac{\partial}{\partial x} k_{\varsigma,i_0} | \leqslant
\frac{9}{\beta_1 - \alpha_1} < \frac{12}{L_1} $. The proof of Lemma~\ref{lemma-about-decomposition} is now complete.

\end{proof}

\subsection{Theorem~\ref{Main-Thm-local-case} - the higher-dimensional case}

The proof of Theorem~\ref{Main-Thm-local-case} for arbitrary
dimension relies on the $2$-dimensional case, and on the following
proposition, the proof of which we postpone to
Subsection~\ref{proof-of-Ck-bound-lemma}.
\begin{proposition} \label{Ck-bound-lemma}
There is a finite family of functions $ {\mathcal F } \subset
C^{\infty}_c(W) $, such that:
\begin{enumerate}
 \item [(i)] Any $ f \in C^{\infty}_{c}(W) $ that can be represented as a product $ f(q,p) = \prod_{i=1}^{n} f_i(q_i,p_i) $,
 for some $ f_i \in C^{\infty}_{c}(I^2) $, satisfies that $ \| f \|_{{ \mathcal F }, \, max} \leqslant C \| f \|_{\infty} $, for some  constant $C$.
 \item [(ii)] For any $ f \in C^{\infty}_{c}(W) $, one has $ \| f \|_{{ \mathcal F }, \, max} \leqslant C \| f \|_{C^{2n+1}} $, for some  constant $C$.
\end{enumerate}
\end{proposition}

\begin{remark} \label{rmk-about-the-collection-of-functions}
{\rm In what follows, we fix ${\mathcal F}$ to be the collection of
functions given by Proposition~\ref{Ck-bound-lemma} above. Moreover,
in order to simplify the presentation, we shall use $ x_1 = q_1, x_2
= p_1, ...,x_{2n-1}=q_n,x_{2n} = p_n $, as another notation for the
coordinates of a point $ x = (q_1,p_1,...,q_{2n},p_{2n})$ in the
$2n$-dimensional cube $W=(-L,L)^{2n}$. }
\end{remark}
\begin{proof}[{\bf Proof of Theorem~\ref{Main-Thm-local-case} (\it the higher dimensional case)}]
For simplicity, the proof of the theorem is divided into two steps:

\noindent {\bf Step I  ({\it Decomposing the function):}} We
consider a smooth function $ r : [-1,1] \rightarrow \mathbb{R} $,
satisfying:
\begin{equation*}
r(t) = \left\{
\begin{array}{rl}
1 & \text{for } t \in [-{\frac 1 3},{\frac 1 3}],\\[0.05in]
0 & \text{for } t \in  [-1,-{\frac 2 3}] \cup [{\frac 2 3},1],
\end{array} \right.
\end{equation*}
and such that $ \sum_{i \in \mathbb{Z}} r(t+i) = 1 $, and $\|r
\|_{\infty}=1$. For any $ \epsilon > 0 $, we denote $$ {\cal
R}^{\epsilon}(x)={\cal R}^{\epsilon}(x_1,x_2,...,x_{2n}) =
\prod_{i=1}^{2n} r \Bigl (\frac{x_i}{\epsilon} \Bigr) $$ Clearly,
one has
 $ \sum_{v \in \epsilon \mathbb{Z}^{2n}} {\cal
R}^{\epsilon}(x-v) = \Id(x) $. Moreover, for a sufficiently small
$\epsilon
>0$,
and a point $w \in { \mathfrak X}:= \{ 0,1,2,3 \}^{2n}$, we consider a finite grid  $\Gamma^{\epsilon}_{w}  \subset W$ given by 
\begin{equation} \label{def-of-the-grid-Gamma} \Gamma^{\epsilon}_{w} =
\epsilon w + 4 \epsilon \mathbb{Z}^{2n} \cap (-L+
3\epsilon,L-3\epsilon)^{2n} \end{equation} 
Furthermore, we define a partition function ${\cal
R}_{w}^{\epsilon}(x)$ by: $$ {\cal R}_{w}^{\epsilon}(x) = \sum_{v
\in \Gamma^{\epsilon}_w} {\cal R}^{\epsilon}(x-v)$$ Note that $
\sum_{w \in {\mathfrak X}} {\cal R}_{w}^{\epsilon}(x) = \Id(x) $ for
any $ x \in (-L+4\epsilon,L-4\epsilon)^{2n} $. Next, consider an
arbitrary function $ f \in C^{\infty}_{c}(W)$. Take $ \epsilon_0 > 0
$ with $ supp \, (f) \subset (-L+4\epsilon_0,L-4\epsilon_0)^{2n} $,
and fix $ \epsilon < \epsilon_0 $. For any $ w \in {\mathfrak X}$,
denote $ f_w(x) =
 {\cal R}_{w}^{\epsilon}(x) f(x) $. Note that $$ f(x) =  \sum_{w \in
{\mathfrak X}} f_w(x) $$ Moreover, for a fix $ w \in {\mathfrak X} $
one has \begin{equation} \label{def-of-f_w}  f_w(x) = \sum_{v \in
\Gamma^{\epsilon}_w} {\cal R}^{\epsilon}(x-v)f(x), \end{equation}
where the support of each summand satisfies $$ supp \, \bigl ({\cal
R}^{\epsilon}(x-v)f(x) \bigr ) \subset v +   \Bigl  [- {\frac
{2\epsilon} 3}  ,{\frac {2\epsilon} 3} \Bigr]^{2n}, \ \ {\rm for} \
v \in \Gamma^{\epsilon}_w
$$

\noindent {\bf Step II } ({\it Estimating the norm $\| f
\|_{{\mathcal F}, \, \max}$}):
Fix $ v \in \Gamma^{\epsilon}_w $, and consider the decomposition of
$f \in C^{\infty}_c(W)$ to a Taylor polynomial of order $2n+1$ and a
remainder, around the point $ v $: $$ f(x) = P_{2n+1}^v(x-v) +
R_{2n+1}^v(x-v)$$ It follows  from~($\ref{def-of-f_w}$)
above that $ f_w(x) = g_w(x) + h_w(x) $, where
$$ g_{w}(x) = \sum_{v \in \Gamma^{\epsilon}_w} {\cal
R}^{\epsilon}(x-v) P_{2n+1}^v(x-v) , \ {\rm and}  \ \ h_{w}(x) =
\sum_{v \in \Gamma^{\epsilon}_w} {\cal R}^{\epsilon}(x-v)
R_{2n+1}^v(x-v) $$

\begin{lemma} \label{lemma-C^k-estimate-of-the-reminder}
With the above notations, there is a constant $C = C(n)$ such that
$$ \| h_{w} \|_{C^{2n+1}} \leqslant C \epsilon
\|f\|_{C^{2n+2}}$$
\end{lemma}

\begin{proof}[\bf Proof of Lemma~\ref{lemma-C^k-estimate-of-the-reminder}]
From the fact that the family
$\{ {\cal R}^{\epsilon}(x-v) R_{2n+1}^v(x-v)   \}_{v \in
\Gamma^{\epsilon}_w}$ has mutually disjoint support, and the
definition of the norm $\| \cdot \|_{C^{2n+1}}$, it follows that
there is a constant $C$ (depending on the dimension) such that
\begin{eqnarray*}
\| h_{w}(x) \|_{C^{2n+1}}  & \leq & 
\max_{v \in \Gamma^{\epsilon}_w} \, \| {\cal R}^{\epsilon}(x-v)
R_{2n+1}^v(x-v) \|_{C^{2n+1}} \\ & \leq & C(n) \,  \max_{v \in \Gamma^{\epsilon}_w} \, \Bigl ( \max_{0 \leq k \leq 2n+1} \| {\cal R}^{\epsilon}(x-v) \|_{C^{k}} \, \| R_{2n+1}^v(x-v) \|_{C^{2n+1-k}} \Bigr) 
\end{eqnarray*}
Note that from the definition of $ {\cal R}^{\epsilon}$ it follows
that for every $0 \leq k \leq 2n+1$, one has $$\| {\cal R}^{\epsilon}(x-v) \|_{C^{k}} \leqslant C' \, \epsilon^{-k}, 
$$ for some constant $C'$ (independent of $k$). 
Note moreover, that for $0 \leq k \leq 2n+1,$ 
\begin{equation} \label{eq-about-the-estimate-of-the-reminder} \| R_{2n+1}^v(x-v) \|_{C^{2n+1-k}} \leqslant C'' \, \|f\|_{C^{2n+2}} \, \epsilon^{1+k}, \end{equation}
for some constant $C''$. Indeed, let $\alpha$ be a multiindex with
$|\alpha| = 2n+1-k$, and consider the order-$k$ Taylor's expension
of $\partial^{\alpha}f$
near the point $v$. 
The remainder equals to $\partial^{\alpha} R_{2n+1}^v(x-v)$, and the
estimate~$(\ref{eq-about-the-estimate-of-the-reminder})$ follows
from the standard bound on the size of the remainder. This completes
the proof of the lemma.
\end{proof}

\begin{corollary} \label{cor-about-max-norm-of-the-reminder} {\rm
From Proposition~\ref{Ck-bound-lemma} (ii), and
Lemma~\ref{lemma-C^k-estimate-of-the-reminder}, we conclude that:}
\begin{equation} \label{eq-estimate-on-h_j} \| h_w \|_{{\mathcal F}, \, max} \leqslant C \epsilon \|f\|_{C^{2n+2}}, \ {\rm  for \ some \ constant} \ C=C(n)  \end{equation}
 \end{corollary}
To complete the proof of the theorem we shall need the following
proposition:
\begin{proposition} \label{estimate-the-max-norm-of-the-polynomial-part}
There is a constant $C = C(n)$ such that
 \begin{equation} \label{eq-estimate-of-g_j} \| g_{w}
\|_{{\mathcal F}, \, max} \leqslant C \bigl ( \sum_{i=0}^{2n+1}
\|f\|_{C^i} \epsilon^i \bigr) \end{equation}
\end{proposition}
Postponing the proof of Proposition~\ref{estimate-the-max-norm-of-the-polynomial-part} to
Subsection~\ref{subsection-estimate-max-norm-of-taylor}, we first complete the proof of Theorem~\ref{Main-Thm-local-case}.
From~$(\ref{eq-estimate-on-h_j})$ and~$(\ref{eq-estimate-of-g_j})$,
letting $\epsilon \rightarrow 0$,  we conclude that
$$ \| f \|_{{\mathcal F}, \, max} \leq C \|f\|_{\infty},$$
for some absolute constant $C$, and the proof is complete.
\end{proof}

\subsubsection{Proof of Proposition~\ref{Ck-bound-lemma}} \label{proof-of-Ck-bound-lemma}

\begin{proof}[{\bf Part 
(i):}] Let $W = \prod_{i=1}^n W^2_i$, where $W^2_i=(-L,L)^2 \subset
{\mathbb R}^2(q_i,p_i)$, and denote by
 ${\mathcal F}^{2} = \{ {\mathfrak f_0},{\mathfrak f_1},{\mathfrak f_2} \}$  the collection of functions constructed in
the proof of Theorem~\ref{Main-Thm-local-case} in the 2-dimensional
case. For any multi-index $\beta = (l_1,\ldots,l_n) \in {\mathfrak
X}':= \{0,1,2 \}^n $, we set ${\mathfrak f}_{\beta}(q,p) =
\prod_{k=1}^n {\mathfrak f}_{l_k}(q_k,p_k) $. In what follows we
denote by ${\mathcal F}$ the set 
$\{ {\mathfrak
f}_{\beta} \, ; \, \beta \in {\mathfrak X}' \}$.

Consider $f \in C_c^{\infty}(W^{2n})$ of the form $f(q,p)=
\prod_{i=1}^{n} f_i(q_i,p_i)$, where $f_i \in C_c^{\infty}(W_{i})$.
Let $ \epsilon > 0 $. From the proof of
Theorem~\ref{Main-Thm-local-case} in the 2-dimensional case it
follows that there exists functions $ f_{i,k} \in {\mathcal
L}_{{\mathcal F}^2} $, $ i=1,2,...,n $; $ k \in \mathbb{N} $, such
that $ f_{i,k} \xrightarrow[]{k \rightarrow \infty }
 f_i $ in the $ C^\infty$-topology, and such that $ \| f_{i,k} \|_{{\mathcal L}_{{\mathcal
F}^2}} < \| f_i \|_{{\mathcal F}^2 , \,  max} + \epsilon $. Next,
for every $1 \leq i \leq n$ and $k \in {\mathbb N}$, we decompose
\begin{equation} \label{def-of-f_ik}
 f_{i,k} = \sum_{j,l} c_{i,k}^{j,l} (\Phi_{i,k}^{j,l})^* {\mathfrak
f}_{l} , \end{equation} where $\Phi_{i,k}^{j,l} \in {\rm Ham}_c(W_i,\omega)$; $l \in \{0,1,2\}$, and, \begin{equation}
\label{eq-estimate-of-c_ik}  \sum_{j,l} | c_{i,k}^{j,l} | < \|
f_{i,k} \|_{{\mathcal L}_{{\mathcal F}^2}} + \epsilon
\end{equation} Denote $ f^k(q,p) = \prod_{i=1}^{n} f_{i,k}(q_i,p_i)
$. Clearly,  $ f^k \xrightarrow[]{k \rightarrow \infty }
f \in C^{\infty}_{c}(W)$ in the $ C^{\infty}$-topology. Moreover,
from~$(\ref{def-of-f_ik})$ it follows that $$ f^k = \sum_{\substack{
\beta = (l_1,...,l_n) \\
\gamma = (j_1,...,j_n)}}
c_{k}^{\gamma,\beta} (\Phi_{k}^{\gamma,\beta})^{*}{\mathfrak
f}_{\beta} ,$$ where $$ c_{k}^{\gamma,\beta} = \prod_{i=1}^{n}
c_{i,k}^{j_i,l_i},  \ {\rm and \ \  }
\Phi_{k}^{\gamma,\beta}(q_1,p_1,...,q_n,p_n) = \left
(\Phi_{1,k}^{j_1,l_1}(q_1,p_1),...,\Phi_{n,k}^{j_n,l_n}(q_n,p_n)
\right)$$
This shows that $ f^k \in {\mathcal L}_{\mathcal F} $, and moreover
that
\begin{equation} \label{estimate-of-LF-norm-of-fk}
\begin{split}
\| f^k \|_{{\mathcal L}_{\mathcal F}}  &  \leqslant \sum_{\substack{
\beta = (l_1,...,l_n) \\
\gamma = (j_1,...,j_n)}} |c_{k}^{\gamma,\beta}| = \prod_{i=1}^{n}
\left( \sum_{j_i,l_i} |c_{i,k}^{j_i,l_i}| \right) < \prod_{i=1}^{n}
\left( \| f_{i,k} \|_{{\mathcal L}_{{\mathcal F}^2}}
 + \epsilon \right) \\
& \leqslant  \prod_{i=1}^{n} \left( \| f_i \|_{{\mathcal F}^2 , \,
max} + 2\epsilon \right)
\end{split}
\end{equation}

Recall, that from the proof of Theorem~\ref{Main-Thm-local-case} in
the 2-dimensional case one has $$ \| f_i \|_{{\mathcal F}^2 , \,
max} \leqslant C \| f_i \|_{\infty} ,$$ for some absolute constant $
C$. Combining this with~$(\ref{estimate-of-LF-norm-of-fk})$ we
conclude that $$\| f^k \|_{{\mathcal L}_{\mathcal F}} \leqslant
\prod_{i=1}^{n} \left( C \| f_i \|_{\infty}  + 2\epsilon \right),
$$ and therefore $$ \| f \|_{{\mathcal F} , \,  max} \leqslant
\liminf_{k \rightarrow \infty} \| f^k \|_{{\mathcal L}_{\mathcal F}}
\leqslant \prod_{i=1}^{n} \left( C \| f_i \|_{\infty}  + 2\epsilon
\right) $$ In particular, for any $\epsilon>0$, one has
$$ \| f \|_{{\mathcal F} , \,  max} \leqslant \prod_{i=1}^{n}
\left( C \| f_i \|_{\infty}  + 2\epsilon \right) $$ Taking $
\epsilon \rightarrow 0 $, we obtain $$ \| f \|_{{\mathcal F} , \,
max} \leqslant C^n \prod_{i=1}^{n} \| f_i \|_{\infty} = C^n \| f
\|_{\infty} $$ This completes the proof of part (i) of
Proposition~\ref{Ck-bound-lemma}.
\end{proof}

For the proof of the second part of Proposition~\ref{Ck-bound-lemma}
we shall need the following preliminaries. Let $f$ be an integrable
function on the $m$-dimensional torus ${\mathbb T}^m$, and denote
its Fourier coefficients by
$$ \hat f_r =  {\frac 1 {(2 \pi)^m}}  \int_{{\mathbb T}^m} f(t) \, e^{ir \cdot t} \, dt,$$
where $r = (r_1,\ldots,r_m) \in {\mathbb Z}^m$, and $t =
(t_1,\ldots,t_m) \in {\mathbb T}^m$. We denote the $j^{th}$-partial
sum of the Fourier series of $f$ by
$$ S_j(f,t) = \sum_{ \max |r_l| \leq j} \hat f_r \, e^{i r \cdot t} $$
The next lemma is a well known result in Fourier analysis.
\begin{lemma} \label{lemma-about-approx-of-Fourier} Let $f \in C^{\infty}({\mathbb T}^m)$. Then 
$S_j(f)\xrightarrow[]{j \rightarrow \infty } f$ in the
$C^{\infty}$-topology and
\begin{equation} \label{eq-about-estimate-of-fourier-coeff} \sum_{r \in {\mathbb Z}^{m}} |\hat f_r| \leq A \| f \|_{C^{2n+1}}, \end{equation}
for some universal constant $A$.
\end{lemma}

\begin{proof}[{\bf Proof of Lemma~\ref{lemma-about-approx-of-Fourier}}]
 The fact that $S_j(f)\xrightarrow[]{j \rightarrow \infty } f$ in the $C^{\infty}$-topology
follows, e.g., from Theorem 33.7 in Section 79 of~\cite{Kor}, and
the fact that $\partial^{\alpha}S_j(f) = S_j(\partial^{\alpha}f)$
for every multi-index $\alpha$ and $j\geq 0$. For the
estimate~$(\ref{eq-about-estimate-of-fourier-coeff})$, we use
 Lemma 9.5 in Section 79 of~\cite{Kor} to obtain the following
upper bound for the Fourier coefficients:

\begin{equation} \label{eq-upper-bound-for-Fourier-coeff} |\hat f_r | \leq A_1 \, {\frac {\|f\|_{C^{2n+1}}} {\|r\|^{2n+1}} } \ \ {\rm for \ all} \ r \neq 0, \end{equation}
for some constant $A_1$. 
From this we conclude that
$$   \sum_{r \in {\mathbb Z}^m} |\hat f_r|  \leq  A_2 \|f\|_{C^{2n+1}} \, \int_{S^{2n-1}} \int_1^{\infty} \rho^{-2n-1} \rho^{2n-1} \, d \rho \, d \theta  \leq  A_3 \, \|f\|_{C^{2n+1}}, $$
where $A=A_3$ is a constant which depends solely on the dimension.
\end{proof}

\begin{remark} \label{rmk-about-fourier-lemma}  {\rm We remark that Lemma~\ref{lemma-about-approx-of-Fourier} holds (with different constants) for any torus of the form
$T^m = ({\mathbb R}/ a {\mathbb Z})^m$, where $a>0$. Moreover, the lemma holds if
instead of the basis $\{ e^{{\frac {2\pi i} a}  rt} \}$, we  choose
the trigonometric basis consists of products of $\{ \cos({{\frac
{2\pi} a}  r_i t_i})\}$ or $\{\sin({{\frac {2\pi} a}  r_i t_i})\}$
for $i=1,\ldots,m$.}
 \end{remark}
We now turn  to complete the proof of the second part of
Proposition~\ref{Ck-bound-lemma}:
\begin{proof}[{\bf Proposition~\ref{Ck-bound-lemma}, Part 
(ii):}] Let $f \in C_c^{\infty}(W)$. By gluing together the boundary
of the cube $W$ in an appropriate way, we obtain a well defined
smooth
function 
on the torus $T^{2n} = ({\mathbb R}/ 2L {\mathbb Z})^{2n}$, which by
abuse of notation we still denote by $f$.
 We apply
Lemma~\ref{lemma-about-approx-of-Fourier} to the function $f$ (note
the comment regarding the trigonometric basis in
Remark~\ref{rmk-about-fourier-lemma}). We order the trigonometric
basis in Remark~\ref{rmk-about-fourier-lemma} by $\{ e_{k}
\}_{k=1}^{\infty}$. Note that each $e_k$ is a product function with
$\| e_k \|_{\infty} = 1$. Denoting the corresponding Fourier sums of
$f$ by $S_k = \sum_{i=1}^k c_i e_i$. We have $S_k \rightarrow f$ in
the $C^{\infty}$-topology and $ \sum_{k=1}^{\infty} |c_k| \leqslant
A \|f\|_{C^{2n+1}} $ for some $ A = A(n) $. We turn back to the
situation where we consider $f$ defined on $ W $. Take any smooth
cutoff function $\rho : W \rightarrow \mathbb{R} $, which equals $1$
on $ supp(f) $, equals $0$ near the boundary $ \partial W $, and
which has $ \| \rho \|_{\infty} = 1 $ (one can easily find such $
\rho $, since $ supp(f) \subset W $). Then we have $ \rho S_k =
\sum_{i=1}^k c_i \rho e_i \rightarrow \rho f = f $ in $
C_c^\infty(W) $, in the $ C^\infty $ topology as well. Moreover, the
functions $ \{ \rho e_k \}$ are product functions with $ \| \rho e_k
\|_{\infty} \leqslant 1 $. From part (i) or
Proposition~\ref{Ck-bound-lemma}, and
Lemma~\ref{lemma-about-approx-of-Fourier}, it follows that for a
suitable collection $ \mathcal{F} $, one has  $$  \| S_k
\|_{{\mathcal F} , \, max} \leqslant   \sum_{i=1}^k |c_i| \| \rho
e_i \|_{{\mathcal F} , \,  max} \leqslant C \sum_{i=1}^k |c_i|
\leqslant CA \|f\|_{C^{2n+1}} .$$ Hence, from
Remark~\ref{rmk-about-max-norm} we conclude that $$  \| f
\|_{{\mathcal F} , \,  max} \leqslant CA \|f\|_{C^{2n+1}} .$$ The
proof of the second part of the proposition is now complete.

\end{proof}

\subsubsection{Proof of Proposition~\ref{estimate-the-max-norm-of-the-polynomial-part}} \label{subsection-estimate-max-norm-of-taylor}
For any multi-index $ \alpha = (i_1,i_2,...,i_{2n}) $, where $ |
\alpha | \leqslant 2n+1 $, denote $$ g^{\alpha}_{w}(x) = \sum_{v =
(v_1,v_2,...,v_{2n}) \in \Gamma^{\epsilon}_w}  \frac{1}{i_1 ! i_2 !
... i_{2n}!} \frac {\partial f^{|\alpha|}} {\partial x_1^{i_1}
\partial x_2^{i_2} ...
\partial x_{2n}^{i_{2n}} }(v) \Big ( \prod_{j=1}^{2n}
(x_j-v_j)^{i_j} \Big ) {\cal R}^{\epsilon}(x-v)  $$ Note that the
function $ g_w $ is the sum of $ g^{\alpha}_{w} $, for $ \alpha =
(i_1,i_2,...,i_{2n}) $ with $ | \alpha | \leqslant 2n+1 $.
Note moreover that each summand of $ g^{\alpha}_{w} $ is a constant
multiple of the function $$  \varXi_{\alpha} (x-v):= \Big (
\prod_{j=1}^{2n} (x_j-v_j)^{i_j} \Big ) {\cal R}^{\epsilon}(x-v)
,$$ where $$ \varXi_{\alpha}(x) = x_1^{i_1} x_2^{i_2} ...
x_{2n}^{i_{2n}} {\cal R}^{\epsilon}(x) =  \prod_{l=1}^{n}
q_l^{i_{2l-1}}  p_{l}^{i_{2l}} \, r \Bigl (\frac{q_l}{\epsilon}
\Bigr) r \Bigl (\frac{p_l}{\epsilon} \Bigr)$$ We shall need the
following lemma which will be proven in
Subsection~\ref{subsection-proof-of-lemma-about-max-norm-of-taylor}
\begin{lemma} \label{ineq-for-sum-of-translations-lemma}
Let $\xi \in C_c^{\infty} \left ( (-\epsilon,\epsilon)^{2n} \right)$
be a compactly supported smooth function which can be represented as
a product $\xi = \prod_{j=1}^n \xi_j(q_j,p_j)$, where $\xi_j \in
C_c^{\infty} \left ((-\epsilon,\epsilon)^2 \right)$. Then, for every
function $H(x) = \sum_{v  \in \Gamma^{\epsilon}_w} a_v \xi(x-v)$,
where $a_v$ are real coefficients and $\Gamma^{\epsilon}_w$ is the
grid defined in~$(\ref{def-of-the-grid-Gamma})$, one has $$ \| H
\|_{{\mathcal F}, \, max} \leqslant C
\|H\|_{\infty}, \ {\it for \ some \ absolute \ constant \ } C$$ 
\end{lemma}
Applying Lemma~\ref{ineq-for-sum-of-translations-lemma}, with $\xi =
\varXi_{\alpha}$, to the function $H=g^{\alpha}_w$, we conclude that
\begin{eqnarray*}  \| g^{\alpha}_w \|_{{\mathcal F}, \, max}  & \leqslant & C \| g^{\alpha}_w \|_{\infty} \leqslant \frac{C}{i_1 ! i_2 ! \ldots i_{2n}!} \,    \| \varXi_{\alpha}  \|_{\infty} \, \max_{v \in \Gamma^{\epsilon}_w} \,
 \frac{\partial f^{|\alpha|}}{\partial x_1^{i_1} \partial x_2^{i_2} \ldots \partial x_{2n}^{i_{2n}} }(v)  \\  & \leqslant &
C \,  \| \varXi_{\alpha}  \|_{\infty} \|f \|_{C^{|\alpha|}}
\end{eqnarray*}
Since $\|r \|_{\infty}=1$, and $supp(r) \subset
(-\epsilon,\epsilon)$, it follows that $\|\varXi_{\alpha}
\|_{\infty} \leq
\epsilon^{|\alpha|}$. Thus, we obtain 
$$ \| g^{\alpha}_w \|_{{\mathcal F}, \, max} \leqslant C \, \epsilon^{|\alpha|} \,  \|f \|_{C^{|\alpha|}} , $$
and hence $$ \| g_w \|_{{\mathcal F}, \, max} \leqslant \sum_{ |
\alpha | \leqslant  2n+1 } C \, \epsilon^{|\alpha|} \,  \|f
\|_{C^{|\alpha|}} \leqslant C' \sum_{k=0}^{2n+1} \, \epsilon^{k} \,
\|f \|_{C^{k}} $$ This completes the proof of
Proposition~\ref{estimate-the-max-norm-of-the-polynomial-part}. \qed
\subsubsection{Proof of Lemma~\ref{ineq-for-sum-of-translations-lemma}} \label{subsection-proof-of-lemma-about-max-norm-of-taylor}

Note first that the grid $ \Gamma^{\epsilon}_w = \epsilon w + 4
\epsilon \mathbb{Z}^{2n} \cap ( -L + 3\epsilon, L - 3\epsilon)^{2n} $  admits a decomposition into the product $
\Gamma^{\epsilon}_w = \prod_{i=1}^{n} \gamma_i $,
where $ \gamma_i = \gamma^{\epsilon,w}_{i}
\subset (-L+3\epsilon,L-3\epsilon)^2 \subset (-L,L)^2 $ are grids on the plane.
Next, let $H$ be as in
Lemma~\ref{ineq-for-sum-of-translations-lemma}. Given a bijection $
\tau : \Gamma^{\epsilon}_w \rightarrow \Gamma^{\epsilon}_w $, we
denote
$$H_{\tau}(x) = \sum_{v  \in \Gamma^{\epsilon}_w} a_{\tau(v)}
\xi(x-v) $$

\begin{lemma} \label{permutation-of-cubes-lemma}
For any bijection $ \tau : \Gamma^{\epsilon}_w \rightarrow
\Gamma^{\epsilon}_w $, one has 
$ \| H_{\tau} \|_{{\mathcal F}, \, max} = \| H \|_{{\mathcal F}, \,
max} $.
\end{lemma}

\begin{proof}[{\bf Proof of Lemma~\ref{permutation-of-cubes-lemma}}]
It is not hard to check that every bijection 
$ \tau : \Gamma^{\epsilon}_w \rightarrow \Gamma^{\epsilon}_w $, can
be written as a product of transpositions that interchange two
neighboring points of $ \Gamma^{\epsilon}_w $ (here, by neighboring
points we mean $
v',v'' \in \Gamma^{\epsilon}_w $, such that 
$ |v'-v''| = 4\epsilon $). Therefore it is enough to prove the lemma
for the case of such a transposition.

Let $ v'=(z_1',...,z_{n}'), v''=(z_1'',...,z_{n}'') \in \Gamma^{\epsilon}_w $ be a pair of neighboring points, where $ z_i' , z_i'' \in \gamma_{i} $ for $ i=1,2,...,n $. There exists $ 1 \leqslant k \leqslant n $, such that $ z_i' = z_i'' $ for $ i \neq k $, and moreover
$z_k'' = z_k' \pm 4\epsilon $ or $ z_k'' = z_k' \pm 4\epsilon i $. The union of the neighboring squares $ Q' := z_k' + [-\epsilon,\epsilon]^2 $, and  $ Q'' := z_k'' + [-\epsilon,\epsilon]^2 $ is a rectangle $ S = Q' \cup Q'' $. Since the support $ supp( \xi_k ) \subset (-\epsilon,\epsilon)^2 $, there exists $ 0 < \epsilon_1 < \epsilon $, such that $ supp( \xi_k ) \subset [-\epsilon_1,\epsilon_1]^2 $. Looking on $ Q_1'= z_k' + [-\epsilon_1,\epsilon_1]^2, Q_1''= z_k'' + [-\epsilon_1,\epsilon_1]^2 \subset int(S) $, one can clearly move $ Q_1' $ to $ Q_2' $ and $ Q_2' $ to $ Q_1' $ simultaneously, using affine translations, such that at every moment the images of $ Q_1', Q_2' $ will not intersect, and are contained in $ int(S) $. Moreover, this can be done by a smooth Hamiltonian isotopy $ \Phi_{K_k}^t $, supported in $ S $, where $ K_k(t,z_k) : [0,1] \times W_k \rightarrow \mathbb{R} $ is the Hamiltonian that generates this isotopy, and such that we have $ supp(K_k(t,\cdot)) \subset int(S) $ for all $ t \in [0,1] $. For any $ j \neq k $, $ 1 \leqslant j \leqslant n $ consider a smooth function $ K_j(z_j): W_j \rightarrow \mathbb{R} $ such that $ K_j(z_j) = 1 $ for $ z_j \in z_j' + [-\epsilon,\epsilon]^2 $ and $ K_j(z_j) = 0 $ for $ z_j \in W_j \setminus ( z_j' + [-2\epsilon,2\epsilon]^2 ) $.
Now define a Hamiltonian $ K : [0,1] \times W \rightarrow \mathbb{R} $ by
$$ K(t;z_1,z_2,...,z_n) = K_{k}(t,z_k)
\prod_{\substack{
1 \leqslant j \leqslant n\\
j \neq k}}
K_{j}(z_j)
$$
Note that $ K(t;z_1,z_2,...,z_n) = K_{k}(t,z_k) $ for $$ z = (z_1,...,z_n) \in U_1 := \prod_{j=1}^{k-1} (z_j' + [-\epsilon,\epsilon]^2 ) \times S \times \prod_{j=k+1}^{n} (z_j' + [-\epsilon,\epsilon]^2 ).$$ Moreover, $ U_1 $ is invariant under the flow $ \Phi_K^t $, and  $$ \Phi_K^t (z_1,...,z_n) = (z_1,...,z_{k-1},\Phi_{K_k}^t(z_k),z_{k+1},...,z_n) ,$$ for any $ z = (z_1,...,z_n) \in U_1 $. In particular, $ \Phi_K^1 (z) = z + v'' - v' $ for $ z \in v' + [-\epsilon,\epsilon]^{2n} $, and $ \Phi_K^1 (z) = z + v' - v'' $ for $ z \in v'' + [-\epsilon,\epsilon]^{2n} $. Furthermore, for $$ U_2 := \prod_{j=1}^{k-1} (z_j' + [-2\epsilon,2\epsilon]^2 ) \times S \times \prod_{j=k+1}^{n} (z_j' + [-2\epsilon,2\epsilon]^2 ) $$ we have that $ supp(K(t,\cdot)) \subset U_2 $ for all $ t \in [0,1] $. Therefore, since $ (v + [-\epsilon,\epsilon]^{2n}) \cap U_2 = \emptyset $ for all $ v \in \Gamma^{\epsilon}_w \setminus \{ v',v'' \} $, we conclude that $ \Phi_K^1 (z) = z $ for $ z \in v + [-\epsilon,\epsilon]^{2n} $ for any $ v \in \Gamma^{\epsilon}_w \setminus \{ v',v'' \} $. Hence if $ \tau : \Gamma^{\epsilon}_w \rightarrow \Gamma^{\epsilon}_w $ is a transposition that interchanges $ v' $ with $ v'' $, we conclude that $ H_\tau = (\Phi^1_K)^* H $. Therefore we conclude $$ \| H_{\tau} \|_{{\mathcal F}, \, max} = \| H \|_{{\mathcal F}, \, max} .$$
\end{proof}

\begin{proof}[{\bf Proof of Lemma~\ref{ineq-for-sum-of-translations-lemma}}]
Consider the decomposition $ \Gamma^{\epsilon}_w = \prod_{i=1}^{n}
\gamma_i $, and write each $ \gamma_i $ explicitly as $ \gamma_i = \{
z_{i,1}, ... , z_{i,N_{i}} \} \subset (-L,L)^2 $. We order each set
$ \gamma_i $ by setting $  z_{i,1} < z_{i,2} < ... < z_{i,N_{i}} $,
for each $ i $, and consider the lexicographic order $ \prec $ on $
\Gamma^{\epsilon}_w $ induced by these orders. We can arrange all the
elements of $ \Gamma^{\epsilon}_w $ by increasing order $$ v_1
\prec v_2 \prec ... \prec v_N ,$$ where $ N = \prod_{i=1}^{n} N_i$.
Take a bijection $ \tau : \Gamma^{\epsilon}_w \rightarrow
\Gamma^{\epsilon}_w $ such that
$$ a_{\tau(v'')} \leqslant a_{\tau(v')} \ {\rm  if \ and \ only \ if \ } v' \preceq v'', \ {\rm where} \  v',v'' \in
\Gamma^{\epsilon}_w, $$
and rewrite $H_{\tau}(x)=\sum_{v  \in
\Gamma^{\epsilon}_w} a_{\tau(v)} \xi(x-v)$ as
\begin{equation} H_{\tau}(x) = 
\sum_{j=1}^{N} b_{j}
\xi(x-v_j), \ {\rm and} \  b_1 \leqslant b_2 \leqslant ... \leqslant b_N  \end{equation}
By Lemma~\ref{permutation-of-cubes-lemma}, one has $
\| H_{\tau} \|_{{\mathcal F}, \, max} = \| H \|_{{\mathcal F}, \,
max} $.
Next, write \begin{equation} H_{\tau}(x) = b_N K_{N}(x) + \sum_{j=1}^{N-1}
(b_{j} - b_{j+1}) K_{j}(x), \end{equation} where $ K_{j}(x) = \sum_{l=1}^{j}
 \xi(x-v_j) $. Also set $ K_0(x) = 0 $.

\begin{equation}
\begin{split}
 \| H_{\tau}
\|_{{\mathcal F}, \, max}  &  \leqslant
|b_N| \| K_{N}(x) \|_{{\mathcal F}, \, max} + \sum_{j=1}^{N-1}
|b_{j} - b_{j+1}| \| K_{j} \|_{{\mathcal F}, \, max} \\ & \leqslant
|b_N| \| K_{N}(x) \|_{{\mathcal F}, \, max} + \sum_{j=0}^{N-1} (b_{j+1} - b_{j}) \max_{1 \leqslant j \leqslant N}
\| K_j \|_{{\mathcal F}, \, max} \\[0.04in]
& =  |b_N| \| K_{N}(x) \|_{{\mathcal F}, \, max} + (b_N - b_1) \max_{1 \leqslant j \leqslant N} \| K_j \|_{{\mathcal F}, \, max} \\
& \leqslant 3 \, \Bigl ( \max_{v \in \Gamma^{\epsilon}_w} |a_{v}| \Bigr)
\max_{1 \leqslant j \leqslant N} \| K_j \|_{{\mathcal F}, \, max}
\end{split}
\end{equation}

Next, consider some $ K_j $, where $ 1 \leqslant j \leqslant N $. There
exist a unique sequence $$ j_0 = 0 \leqslant j_1 \leqslant j_2
\leqslant ... \leqslant j_{n-1} \leqslant j_n = j ,$$ such that for
any $ 1 \leqslant m \leqslant n $ we have $ \prod_{l=m+1}^{n} N_{l}
\, | \, j_{m}-j_{m-1} $, and we have $$ k_l := \frac{j_{m}-j_{m-1}}{
\prod_{l=m+1}^{n} N_l } < N_m .$$ Here we mean $ \prod_{l=n+1}^{n}
N_{l} = 1 $. Take any $ 1 \leqslant m \leqslant n $. Then provided $
j_{m-1} < j_{m} $, we can write $$ \xi^m(z) := K_{j_m} - K_{j_{m-1}}
= \prod_{l=1}^{n} \xi^m_l (z_l), $$ where we have $$ \xi^m_l(z_l) =
\xi_l(z_l - z_{l,k_l}) \text{ , for } l=1,...,m-1 , $$
$$ \xi^m_m(z_m) = \sum_{i_m = 1}^{k_m} \xi_m(z_m - z_{m,i_m}) ,$$
$$  \xi^m_l(z_l) = \sum_{i_l = 1}^{N_l} \xi_l(z_l - z_{l,i_l}) \text{ , for } l=m+1,...,n .$$
Moreover, for any $ 1 \leqslant m \leqslant  n $ we have $$ \| \xi^m
\|_{\infty} = \prod_{l=1}^{n} \| \xi^m_l \|_{\infty} =
\prod_{l=1}^{n} \| \xi_l \|_{\infty} = \| \xi \|_{\infty} .$$ From
this, and from Proposition~\ref{Ck-bound-lemma}~(i), we conclude
that $$ \|  \xi^m \|_{{\mathcal F}, \, max} \leqslant C \|  \xi^m
\|_{\infty} =  C \|  \xi \|_{\infty} ,$$ for some $ C=C(n) $. We
have $$ K_j = \sum_{m=1}^{n} \xi^m ,$$ hence $$ \| K_j \|_{{\mathcal
F}, \, max} \leqslant \sum_{m=1}^{n} \|  \xi^m \|_{{\mathcal F}, \,
max} \leqslant nC \|  \xi \|_{\infty} ,$$ and this holds for any $ 1
\leqslant j \leqslant N $. Therefore we conclude
$$ \| H \|_{{\mathcal F}, \, max} = \| H_{\tau} \|_{{\mathcal F}, \, max}
\leqslant 3 \left( \max_{v \in \Gamma^{\epsilon}_w} |a_{v}| \right)
\max_{1 \leqslant j \leqslant N} \| K_j \|_{{\mathcal F}, \, max}
$$ $$ \leqslant 3nC \left( \max_{v  \in \Gamma^{\epsilon}_w} |a_{v}| \right) \| \xi \|_{\infty} = 3nC \| H \|_{\infty} .$$
The proof of the lemma is now complete.

\end{proof}

\section{Proof of Theorem~\ref{Main-thm}} \label {section-proof-of-thm}

The proof of Theorem~\ref{Main-thm} follows from
Theorem~\ref{Main-Thm-local-case} by a standard partition of unity
argument. For the sake of completeness, we provide the details
below.

As explained in Section~\ref{section-outline}, it is enough to prove
Theorem~\ref{Main-thm} for Ham$(M,\omega)$-invariant pseudo norms on
$C^{\infty}(M)$. Indeed, any Ham$(M,\omega)$-invariant pseudo norm
$\| \cdot \|$ on ${\mathcal A}$ that is continuous in the
$C^{\infty}$-topology, can be naturally extended to a Ham$(M,\omega)$-invariant
pseudo-norm $\| \cdot \|'$ on $C^{\infty}(M)$, which is again continuous in the
$C^{\infty}$-topology, by setting
$$\| f \|' = \| f- M_f \|, \ {\rm where \ } M_f = {\textstyle {\frac
1 {Vol(M)}} \int_M f \omega^n}$$

Consider a Darboux chart $ i:  U \hookrightarrow M $, where $ U
\subset (\mathbb{R}^{2n},\omega_{std}) $ is an open set. Without
loss of generality we assume that the origin of $ \mathbb{R}^{2n} $
lies inside $ U $. Choose some $ L
> 0 $, such that $ W=(-L,L)^{2n} \subset U $. Since $ i(W) \subset
M $, we have a natural embedding $ C^{\infty}_{c}(i(W))
\hookrightarrow C^{\infty}(M) $, and therefore any
Ham$(M,\omega)$-invariant pseudo norm $ \| \cdot \| $ on $
C^{\infty}(M) $ restricts to $ C^{\infty}_{c}(i(W)) $. From
Lemma~\ref{lemma-about-max-norm} and
Theorem~\ref{Main-Thm-local-case}, we conclude that (when the norm is continuous in the $C^{\infty}$-topology) there exists a
constant $ C
> 0 $ such that $$ \| f \| \leqslant C \| f \|_{\infty}, \ {\rm  for  \ every \ function
} \ f \in C^{\infty}_{c}(i(W))$$
Next, for any point $ x \in M $
there exists an open neighborhood $ V_x \subset M $, and a smooth
Hamiltonian diffeomorphism $ \Phi_{x} \in {\rm Ham}(M,\omega) $,
such that $ \Phi_{x}(V_x) \subset W $. Consider the open covering $
\bigcup_{x \in M} V_x = M $. The compactness of $M$ allows us to
pass to a finite subcover $ \bigcup_{i=1}^{N} V_{x_i} = M $.
Moreover, one can find a partition of unity $ \{
\rho_1,\rho_2,...,\rho_N\}$, such that for every $ i=1,2,...,N $,
$\rho_i: M \rightarrow \mathbb{R}$ is a smooth positive function
supported in $V_{x_i}$, and
$$ \rho_1 + \rho_2 + ... + \rho_N = \Id_{M} $$ 
Finally, let $ f \in C^{\infty}(M) $, and consider the decomposition
$$ f = \rho_1 f+ \rho_2 f+ ... + \rho_N f$$ Since $\| \cdot \|$ is a Ham$(M,\omega)$-invariant norm, it follows that $$ \| f \|
\leqslant \sum_{i=1}^{N} \| \rho_i f\| = \sum_{i=1}^{N} \|
(\Phi_{x_i}^{-1})^* (\rho_i f) \|$$ Moreover, it follows from the
above that $ supp \left ( (\Phi_{x_i}^{-1})^* (\rho_i f) \right)
\subset W $, and hence
$$ \| (\Phi_{x_i}^{-1})^* (\rho_i f) \| \leqslant C \|
(\Phi_{x_i}^{-1})^* (\rho_i f) \|_{\infty} = C \| \rho_i f
\|_{\infty} \leqslant C \| f \|_{\infty} .$$ Therefore we conclude
that $$ \| f \| \leqslant C' \|f\|_{\infty},$$ where $C'=NC$. The
proof of the theorem is now complete.

\section{Appendix}

Here we prove the claim mentioned in Remark~\ref{Rmk-about-continuity}. More precisely:

\begin{proposition}
 Let $ M $ be a closed symplectic manifold, and let $ \| \cdot \| $ be a norm on the Lie algebra ${\cal A}$ of $ Ham(M,\omega) $. Then, smooth paths $ [0,1] \rightarrow Ham(M,\omega) $ have finite length if and only if the norm $ \| \cdot \| $ is continuous in the $ C^\infty $-topology.
\end{proposition}

\begin{proof}
 The ``if" part of the statement is clear. Let us show the ``only if" part.
 Throughout, we equip $M$ with a Riemmanian metric, and denote
$ \| \cdot \|_{\infty} = \| \cdot \|_{C^0} \leqslant \| \cdot \|_{C^1} \leqslant \| \cdot \|_{C^2} \leqslant ... $
the corresponding $ C^0,C^1,C^2,... $-norms on $ C^{\infty}(M) $.

Let $\| \cdot \|$ be an invariant pseudo-norm on  $C^{\infty}(M)$ 
which is not continuous in the $ C^\infty $-topology. Consider 
two sequences $ \{a_k \}, \{b_k \}$ in the interval $[0,1] $, such that $$ 0 < a_1 < b_1 < a_2 < b_2 < ... < 1 $$
Next, 
let $ c : [0,1] \rightarrow [0,1] $ be a smooth function such that $ c(t) = 0 $ for $ t \in [0,\frac{1}{4}] \cup [\frac{3}{4},1] $, and $ c(t) = 1 $ for $ t \in [\frac{1}{3},\frac{2}{3}] $. For a sequence of smooth functions $ H_k : M \rightarrow \mathbb{R} $, we define a function
$ H: M \times [0,1] \rightarrow \mathbb{R} $ in the following way:
\begin{equation} \label{append-def-of-H}
H(x,t) = \left\{
\begin{array}{lll}
0 & \text{for } t \in [0,a_1] \cup [b_1,a_2] \cup [b_2,a_3] \cup ...,\\[0.1in]
c(\frac{t-a_k}{b_k-a_k}) H_k(x)  & \text{for } t \in [a_k,b_k], \\[0.1in]
0 & \text{for } t=1.
\end{array} \right.
\end{equation}
Note that $ H $ is smooth on $ M \times [0,1) $.
We next show that for a suitable choice of a sequence $ H_k \in C^\infty(M) $, one has $ H(x,t) \in C^\infty(M \times [0,1]) $, and moreover \begin{equation} \label{estimate-for-infinite-length} \int_{0}^{1} \| H(\cdot,t) \| dt = +\infty \end{equation}
Thus, the Hamiltonian flow of $H$ has infinite length with respect to the Finsler
metric $d_{\| \cdot \|}$. Indeed, note that $$ \int_{0}^{1} \| H(\cdot,t) \| dt =
\sum_{k=1}^{\infty} (b_k-a_k) \left( \int_{0}^{1} |c(t)| dt \right) \|H_k\| \geqslant \frac{1}{3} \sum_{k=1}^{\infty} (b_k-a_k) \|H_k\| .$$ Hence, for the estimate~$(\ref{estimate-for-infinite-length}$),
it is enough to choose $H_k$ such that $ \|H_k\| \geqslant \frac{1}{b_k-a_k} $.
Moreover, to ensure that $ H(x,t) $ is smooth in $ M \times [0,1] $,
it is enough to have \begin{equation} \label{estimate-for-infinite-length1} \lim_{t \rightarrow 1} \| \frac{\partial^j}{\partial t^j} H(t,\cdot) \|_{C^m} = 0 , \ \ {\rm for \ any \ } j,m \geqslant 0 \end{equation}
More precisely, assume that $ t \in (a_k,b_k) $. Note that in that case $$ \| \frac{\partial^j}{\partial t^j} H(t,\cdot) \|_{C^m} = \left( \frac{1}{b_k-a_k} \right)^{j} \Bigl | c^{(j)}(\frac{t-a_k}{b_k-a_k}) \Bigr | \|H_k\|_{C^m} \leqslant \left( \frac{1}{b_k-a_k} \right)^{j} \|c\|_{C^j} \|H_k\|_{C^m} .$$ Therefore, to show~$(\ref{estimate-for-infinite-length1})$ it is enough to choose $H_k$ such that
$$ \lim_{k \rightarrow \infty} \left( \frac{1}{b_k-a_k} \right)^{j} \|H_k\|_{C^m} = 0, \ \ {\rm for \ any \ } j,m \geqslant 0 $$
In particular, any sequence 
$H_k \in C^{\infty}(M)$, that for every $k \geqslant 1$ satisfy
\begin{equation} \label{App-some-cond-on-Hk}
\left\{
\begin{array}{ll}
\|H_k\| \geqslant \frac{1}{b_k-a_k},\\[0.1in]
 \| H_k \|_{C^k} \leqslant (b_k-a_k)^k,
\end{array} \right.
\end{equation}
would give rise (via definition~$(\ref{append-def-of-H}$))  to a smooth function $H:  M \times [0,1] \rightarrow {\mathbb R}$, such that $ \int_{0}^{1} \| H(\cdot,t) \| dt = +\infty $.

 Since the norm $ \| \cdot \| $ is assumed to be non-continuous in the $ C^\infty $-topology,
one can always find a sequence $\{H_k\}$ which satisfy~$(\ref{App-some-cond-on-Hk})$.

\end{proof}

\bigskip
\noindent Lev Buhovsky\\
The Mathematical Sciences Research Institute, Berkeley, CA 94720-5070 \\
{\it e-mail}: levbuh@gmail.com
\bigskip

\noindent Yaron Ostrover\\
School of Mathematics, Institute for Advanced Study, Princeton NJ 08540\\
{\it e-mail}: ostrover@math.ias.edu

\end{document}